\documentclass[11pt,reqno]{amsart}
\usepackage{amssymb,latexsym}
\usepackage{amsmath,color}
\usepackage{amsthm}
\usepackage{epsfig}
\usepackage{graphicx}
\usepackage{titletoc}
\numberwithin{equation}{section}

\newtheorem{theorem}{Theorem}[section]
\newtheorem{proposition}[theorem]{Proposition}
\newtheorem{lemma}[theorem]{Lemma}
\newtheorem{corollary}[theorem]{Corollary}

\theoremstyle{definition}

\newtheorem{remark}{Remark}

\def\e{\varepsilon}

\def\u{u_{\lambda_m}}

\newcommand{\ds}{\displaystyle}
\newcommand{\eps}{\varepsilon}

\newcommand{\wmr}{{\rho_0}}
\newcommand{\ueps}{{u}_{\varepsilon}}

\newcommand{\meps}{\rho_{\varepsilon}}
\newcommand{\bigu}{{\psi}_{\varepsilon}}

\newcommand{\dx}{\mathrm{d}x}

\newcommand{\dr}{\mathrm{d}r}
\newcommand{\dss}{\mathrm{d}s}

\newcommand{\hnu}{\pmb{\mathrm{J}}(u_0)}
\newcommand{\jnu}{{\pmb{\mathrm{J}^{*}}}(u_0,\Psi^R(d_0))}
\newcommand{\wjnu}{{\pmb{{\mathrm{J}^{**}}}}(u_0,\Psi^R(d_0))}

\begin{document}
	\title[Singularly perturbed non-local semi-linear problem]{Boundary-layer profile of a singularly perturbed non-local semi-linear problem arising in chemotaxis }
	
	\author{Chiun-Chang Lee}
	\address{Chiun-Chang Lee,~Institute for Computational and Modeling Science, National Tsing Hua University, Hsinchu 30013, Taiwan}
	\email{chlee@mail.nd.nthu.edu.tw}
	\author{Zhian Wang}
	\address{ Zhian Wang,~Department of Applied Mathematics, Hong Kong Polytechnic University, Hung Hom, Kowloon, Hong Kong}
	\email{mawza@polyu.edu.hk}
	\author{ Wen Yang}
	\address{ Wen Yang,~Wuhan Institute of Physics and Mathematics, Chinese Academy of Sciences, P.O. Box 71010, Wuhan 430071, P. R. China}
	\email{wyang@wipm.ac.cn}
	
	\date{}
	
	\maketitle
	
	\begin{abstract}
		This paper is concerned with the following singularly perturbed non-local semi-linear problem
		\begin{equation}\label{h}\tag{$\ast$}
		\begin{cases}
		\varepsilon^2 \Delta u=\dfrac{m}{\int_{\Omega}e^{u}{\dx}}u e^u\quad &\mathrm{in}~\Omega,\\
		u= u_0~&\mathrm{on}~\partial\Omega,
		\end{cases}
		\end{equation}
		which corresponds to the stationary problem of a chemotaxis system describing the aerobic bacterial movement, where $\Omega$ is a smooth bounded domain in $\mathbb{R}^N (N\geq 1)$, $\eps, m$ and $u_0$ are positive constants.  We show that the problem \eqref{h} admits a unique classical solution which is of boundary-layer profile as $\varepsilon \to 0$, where the boundary-layer thickness is of order $\varepsilon$. When $\Omega=B_R(0)$ is a ball with radius $R>0$, we find a refined asymptotic boundary layer profile up to the first-order expansion of $\varepsilon$ by which we find that the slope of the layer profile in the immediate vicinity  of the boundary  decreases with respect to (w.r.t.) the  curvature while the boundary-layer thickness increases {w.r.t.} the curvature.
	\end{abstract}

	\section{introduction}
	Aerobic bacteria often live in thin fluid layers near the air-water interface where the dynamics of bacterial chemotaxis, oxygen diffusion and consumption can be encapsuled in the following mathematical model (see \cite{Tuval})
	\begin{equation}
	\label{1n.sys}
	\begin{cases}
	v_t+\vec{w} \cdot \nabla u=\Delta v-\nabla\cdot  (v\nabla u)\quad&\mathrm{in}~\Omega,\\
	u_t+\vec{w} \cdot \nabla u=D\Delta u-uv&\mathrm{in}~\Omega,\\
	\rho(\vec{w}_t+\vec{w}\nabla \vec{w})+\nabla p=\mu \Delta \vec{w}-v \nabla \phi &\mathrm{in}~\Omega,\\
	\nabla \cdot \vec{w}=0,
	\end{cases}
	\end{equation}
	where $\Omega$ is a smooth bounded domain in $\mathbb{R}^n (n\geq 1)$, $v(x,t)$ and $u(x,t)$ denote the concentration of bacteria and oxygen, respectively, and $\vec{w}$ is the velocity field of a fluid flow governed by the incompressible Navier-Stokes equations with density $\rho$, pressure $p$ and viscosity $\mu$, where $\nabla \phi=V_b g(\rho_b-\rho){\bf z}$ describes the gravitational force exerted by bacteria onto the fluid along the upward unit vector ${\bf z}$  proportional to the bacterial volume $V_b$, the gravitational constant $g$ and the bacterial density $\rho_b$; $D$ is the diffusion rate of oxygen. The system \eqref{1n.sys} describes the chemotactic  movement of bacteria towards the concentration of oxygen which is saturated with a constant $u_0$ at the air-water interface (boundary of $\Omega$) and will be absorbed (consumed) by the bacteria, where both bacteria and oxygen diffuses and are convected with the fluid. Therefore the physical boundary conditions as employed in \cite{Tuval} is the zero-flux boundary condition on $v$ and Dirichlet boundary condition on $u$ as well as no-slip boundary condition on $\vec{w}$, namely
	\begin{equation}\label{bc1}
	\partial_{\nu}v-v\partial_{\nu}u=0,~u=u_0, \ \vec{w}=0\quad \mathrm{on}~\partial\Omega
	\end{equation}
	where $u_0$ is a positive constant accounting for the saturation of oxygen at the air-water interface and $\nu$ denotes the unit outward normal vector to the boundary $\partial\Omega$. The model \eqref{1n.sys}-\eqref{bc1} has been successfully used in \cite{Tuval} to numerically recover the (accumulation) boundary layer phenomenon observed in the water drop experiment reported  in \cite{Tuval}. Later more extensive numerical studies were performed in \cite{CFKLM, LeeKim} for the model \eqref{1n.sys} in a chamber. Analytic study of \eqref{1n.sys} on the water-drop shaped domain as in \cite{Tuval} with physical boundary condition \eqref{bc1} was started with \cite{Lorz1} where the local existence of weak solutions was proved.  Recent works \cite{Peng-Xiang1, Peng-Xiang2} obtained the global well-posedness of a variant of \eqref{1n.sys} in a 3D cylinder with mixed boundary conditions under some additional conditions on the consumption rate. The above-mentioned  appear to the only analytical results of \eqref{1n.sys} with physical boundary conditions \eqref{bc1} in the literature. In the meanwhile, there are many results on the unbounded whole space $\mathbb{R}^N (N\geq 2)$ or bounded domain with Neumann boundary conditions on both $v$ and $u$ as well as no-slip boundary condition on $\vec{w}$ (see earliest works \cite{Duan, Liu, Winkler}).
	
	It should be emphasized that most important finding of the experiment performed in \cite{Tuval} was the boundary-layer formation by bacteria near the air-water interface. Therefore an analytical question is naturally to exploit whether the model \eqref{1n.sys}-\eqref{bc1} will have boundary-layer solutions relevant to the experiment of \cite{Tuval}.  Except some numerical studies recalled above, rigorous analysis on the boundary-layer formation of the model \eqref{1n.sys}-\eqref{bc1} with physical boundary conditions \eqref{bc1} seems unavailable in the literature as far as we know. The purpose of this paper is to make some progress towards the understanding of boundary layer solutions of the concerned system. As the first step we consider a simplified fluid-free aero-taxis model \eqref{1n.sys}-\eqref{bc1} with physical boundary conditions
	\begin{equation}
	\label{1.sys}
	\begin{cases}
	v_t=\Delta v-\nabla\cdot ( v\nabla u)\quad&\mathrm{in}~\Omega,\\
	u_t=D\Delta u-uv&\mathrm{in}~\Omega,\\
	\partial_{\nu}v-v\partial_{\nu}u=0,~u=u_0\quad &\mathrm{on}~\partial\Omega,
	\end{cases}
	\end{equation}
	which resembles a consumption-type chemotaxis system initially appeared in \cite{KS}.
	Even for the simplified system \eqref{1.sys},  due to the lack of effective mathematical tools handling chemotaxis  systems with non-homogeneous Dirichlet boundary conditions, the global well-posedness of \eqref{1.sys} still remains an open question.  When the boundary conditions are changed to Neumann boundary conditions $\partial_\nu v|_{\partial \Omega}=\partial_\nu u|_{\partial \Omega}=0$, some results on the global well-posedness and large-time behavior of solutions to \eqref{1n.sys} have been developed in \cite{FJ, Tao1, Tao2}.
	In this paper, we shall study the boundary layer solutions of the stationary problem of \eqref{1.sys}. Integrating the first equation of \eqref{1.sys} in space with zero-flux boundary condition directly, we find that the bacterial mass is preserved in time, namely
	\begin{equation*}\label{mass}
	\int_\Omega v(x,t)\dx=\int_\Omega v(x,0)\dx:=m
	\end{equation*}
	where $m>0$ denotes the initial bacterial mass.  Therefore the stationary problem of \eqref{1.sys} reads as
	\begin{equation}
	\label{1.steady}
	\begin{cases}
	\Delta v -\nabla\cdot(v\nabla u)=0&\mathrm{in}~\Omega,\\
	D\Delta u-uv=0&\mathrm{in}~\Omega,\\
	\partial_{\nu}v-v\partial_{\nu}u=0,~u=u_0\quad &\mathrm{on}~\partial\Omega,\\
	\int_\Omega v(x)\dx=m.
	\end{cases}
	\end{equation}
	Note the first equation of \eqref{1.steady} can be written as
	$
	\nabla\cdot(v\nabla(\log v-u))=0.
	$
	Then multiplying both sides of this equation by $\log v-u$ and using the zero-flux boundary condition, we find that any solution of \eqref{1.steady} verifies the equation
	\begin{equation*}
	\int_\Omega v|\nabla (\log v-u)|^2\dx=0,
	\end{equation*}
	which gives $v=\lambda e^u$ for some positive constant $\lambda$. Since
	$
	m=\int_{\Omega} v(x)\dx,
	$
	we get $\lambda=\frac{m}{\int_{\Omega}e^u{\dx}}$. Therefore the problem \eqref{1.steady} is equivalent to the following nonlocal semilinear elliptic Dirichlet problem
	\begin{equation}
	\label{1.model}
	\begin{cases}
	\varepsilon^2 \Delta u=\dfrac{m}{\int_{\Omega}e^{u}{\dx}}u e^u\quad &\mathrm{in}~\Omega,\\
	u= u_0~&\mathrm{on}~\partial\Omega,
	\end{cases}
	\end{equation}
	with
	\begin{equation}\label{vu}
	v=\frac{m}{\int_{\Omega}e^u{\dx}} e^u,
	\end{equation}
	where for convenience we have assumed $D=\varepsilon^2$ for $\varepsilon>0$.
	
	The purpose of this paper is threefold: (i) prove the existence and uniqueness of classical solutions of \eqref{1.steady} for any $\eps>0$; (ii) justify that the unique solution obtained in (i) has a boundary-layer profile as $\varepsilon \to 0 $; (iii) find the refined asymptotic structure of boundary-layer profile near the boundary and explore how the (boundary) curvature affects the boundary-layer profile like the steepness and thickness.  The result (i) confirms that the system \eqref{1.sys} has pattern formation, and result (ii) shows that the pattern solution is of a boundary-layer profile as $\varepsilon \to 0$ which roughly provides a theoretical explanation of the accumulation boundary-layer at the water-air interface observed in the experiment of \cite{Tuval}. The result (iii) further elucidates why the boundary layer thickness varies at the air-water interface of water drop with different curvatures observed in the experiment of \cite{Tuval}.
	
	The major difficulty in exploring the above three questions lies in the non-local term in  \eqref{1.steady}.
	To prove the result (i), we first show that the existence of solutions to the nonlocal problem \eqref{1.steady} can be provided by an auxiliary  (local) problem for which we use the monotone iteration scheme along with elliptic regularity theory to get the existence, and then show the uniqueness of \eqref{1.steady}  directly. The boundary-layer profile as $\eps \to 0$ in a general domain $\Omega$ as described in (ii) is justified by the \textcolor{black}{Fermi coordinates (see \cite{dkw} for more background of Fermi coordinates) and the barrier method.} The non-locality in \eqref{1.model} does not seems to bring much troubles for the first two results. It, however, brings considerable difficulties to our third question (iii) concerning the effect of boundary curvature on the boundary-layer profile.  In order to explore the question (iii), we have to have a good understanding of the asymptotic structure of the non-local coefficient $\int_\Omega e^{u_\eps} {\dx}$ which, however, depends on the asymptotic profile of $u$ itself.  Moreover, we have to make the asymptotic expansion as precise as possible so that the role of curvature can be explicitly observed.  This makes the problem very tricky and challenging. With this non-locality, we are unable to gain the necessary understanding of the solution-dependent nonlocal coefficient $\int_\Omega e^udx$ in a general domain $\Omega$. Fortunately when the domain is a ball, we manage to derive {the required} estimates on this nonlocal term and find the refined asymptotic profile of boundary-layer solutions as $\varepsilon \to0$ involving the (boundary) curvature whose role on the boundary-layer steepness and thickness can be explicitly revealed.
	
	Finally, we mention some other  results comparable to the current work. \textcolor{black}{When the nonlinear term $ue^u$ is replaced by the double well type function, including the Allen-Cahn type nonlinearity, the boundary expansion (up to the 2nd order) of the Neumann derivative for the case without the non-local term was obtained by Shibata in \cite{Sh2003,Sh2004}.} While if the first equation of \eqref{1.sys} was replaced by the $v_t=\Delta v- \nabla\cdot ( v\nabla \ln u)$, namely the chemotactic sensitivity is logarithmic, and the Dirichlet boundary condition for $v$ and Robin boundary condition for $u$ are prescribed, the boundary-layer solution of time-dependent problem has been studied in a series works \cite{hwz,hw1,hw2} where the boundary-layer appears in the gradient of $u$ other than $u$ itself.  Very recently, when the boundary condition for $u$ is changed to another physical boundary condition $\partial_\nu u=(\gamma-u(x)) g(x)$ where $\gamma \geq 0$ denotes the maximal saturation of oxygen in the fluid and $g(x)$ is the absorption rate of the gaseous oxygen into the fluid, the following stationary problem corresponding to \eqref{1.sys} with $D=1$
	\begin{equation*}
	\label{1n.model}
	\begin{cases}
	\Delta u=\sigma u e^u\quad &\mathrm{in}~\Omega,\\
	\partial_\nu u=(\gamma-u(x)) g(x)~&\mathrm{on}~\partial\Omega
	\end{cases}
	\end{equation*}
	was considered in \cite{Braukhoff} and the existence of non-constant classical solutions was established, where $\sigma>0$ is a constant. Clearly the nonlocal elliptic problem \eqref{1.model} is very different from the problems mentioned above, and more importantly we focus on the question whether the nonlocal problem \eqref{1.model} admit boundary-layer solutions relevant to the experimental observation in \cite{Tuval}.
	
	The rest of paper is organized as follows. In section 2, we shall {state} the main results on the existence of non-constant classical solutions of \eqref{1.model} (see Theorem \ref{th1.1}), the existence of boundary layer solution as $\eps \to 0$ (see Theorem \ref{th1.2}) and refined asymptotic profile of boundary layer solutions as $\eps$ is small (see Theorem \ref{mainthm}). In section 3, we prove Theorem \ref{th1.1}. In section 4, we prove Theorem \ref{th1.2}. Finally, Theorem \ref{mainthm} is proved in section 5.

	\section{Statement of the main results}
	We shall first prove the existence of a unique solution to \eqref{1.model} and then pass the results to the original steady state problem \eqref{1.steady}. Furthermore, we can show the solution of \eqref{1.steady} is non-degenerate, i.e., the associated linearized problem only admits a trivial (zero) solution. The results are stated in the following theorem.
	\begin{theorem}
		\label{th1.1}
		Let $\Omega$ be a bounded smooth domain in $\mathbb{R}^N(N\geq 1)$ with smooth boundary, and let $m$ and $u_0$ be given positive constants independent of $\varepsilon$. Then, for $\varepsilon>0,$ equation \eqref{1.model} admits an unique classical solution $u\in\mathrm{C}^1(\overline{\Omega})\cap\mathrm{C}^{\infty}(\Omega)$, and hence  the elliptic system \eqref{1.steady} admits a unique solution which is non-degenerate.
	\end{theorem}
	

	Our second result on the problem \eqref{1.model} is the explicit behavior of $u$ near the boundary $\partial\Omega$ when $\varepsilon$ is small. Before statement, we introduce the following notations. Let $\Omega_{\delta}$ be defined by
	\begin{equation*}
	\label{1.def-inner}
	\Omega_{\delta}=\left\{p\in\Omega\mid\mathrm{dist}(p,\partial\Omega)
	>\delta\right\}
	\end{equation*}
	as illustrated in Fig.\ref{fig1}. For example, when $n=1$ and $\Omega=(-1,1)$, then
	$
	\Omega_{\delta}=(-1+\delta,1-\delta).
	$
	When $n=2,$ and $\Omega=B_R(0)$, then $\Omega_{\delta}=B_{R-\delta}(0)$.
	
	
	\begin{figure}[h]
		\includegraphics[width=7cm]{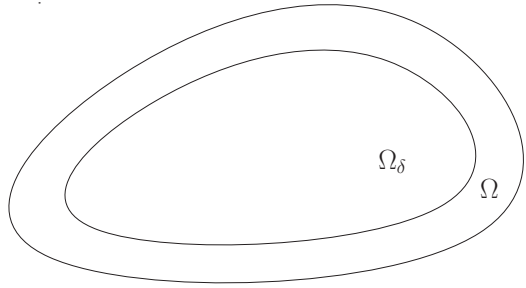}
		\caption{An illustration of $\Omega_\delta$ in $\Omega$.}
		\label{fig1}
	\end{figure}
	
	In the following, we shall give some description on the solution of \eqref{1.model} in general domain as $\e\to0$. Without loss of generality, we may assume $0\in\Omega$ throughout the paper and set
	\begin{equation*}
	\Omega^{\e}=\left\{y\mid \e y\in \Omega\right\}.
	\end{equation*}
	To find the leading order term for the solution of \eqref{1.model}, we define $U_{\e}(y)$ to be the solution of the following equation
	\begin{equation}
	\label{1.def-u1}
	\begin{cases}
	\Delta_y U_{\e}=\frac{m}{|\Omega|}U_{\e} e^{U_{\e}}~&\mbox{in}~\Omega^\e,\\
	U_{\e}=u_0~&\mbox{on}~\partial\Omega^\e.
	\end{cases}
	\end{equation}
	The second result of this paper is as follows
	\begin{theorem}
		\label{th1.2}
		Let $\Omega$ be a bounded domain with smooth boundary. Then there is some non-negative constant $\delta(\e)$ satisfying
		$$\delta(\e)\to0\quad\mathrm{and}\quad \e/\delta(\e)\to0~\mathrm{as}~\e\to 0,$$
		and the unique solution $u_\e$ obtained in Theorem \ref{th1.1} has the following property:
		\begin{equation}
		\label{1.th2-1}
		\lim_{\e\to0}\|u_\e\|_{L^\infty\left(\overline{\Omega_{\delta(\e)}}\right)}=0,
		\end{equation}
		and
		\begin{equation}
		\label{1.th2-2}
		\|u_\e(x)-U_{\e}(x/\e)\|_{L^\infty(\Omega)}=O(\e).
		\end{equation}
	\end{theorem}
	
	In the next result we shall derive that the  boundary-layer thickness is of order $\e.$ Specifically, we have
	
	\begin{corollary}
		\label{th1.3}
		Let $u_\e(x)$ be the solution of \eqref{1.model} and $x_{in}$ be any interior point such that the distance
		from $x_{in}$ to the boundary is of order $\ell_\e$, namely $\mathrm{dist}(x_{in},\partial\Omega) \sim \ell_\e$. Under the same hypothesis as in Theorem \ref{th1.2}, as $0<\e\ll1$, we have:
		\begin{enumerate}
			\item [(1)] if $\lim_{\e\to0}\frac{\ell_\e}{\e}=0$, then $\lim_{\e\to0}u_\e(x_{in})=u_0$;
			\item [(2)] if $\lim_{\e\to0}\frac{\ell_\e}{\e}=L$ with $L\in(0,\infty)$, then $\lim_{\e\to0}u_\e(x_{in})\in(0,u_0)$;
			\item [(3)] if $\lim_{\e\to0}\frac{\ell_\e}{\e}=+\infty$, then
			$\lim_{\e\to0}u_\e(x_{in})=0.$
		\end{enumerate}
	\end{corollary}
	
	Finally we investigate the refined boundary-layer profile of \eqref{1.model} by finding its asymptotic expansion with respect to $\eps$ and exploiting how the boundary curvature affects the boundary layer profile. This is challenging question in general due to the non-locality as discussed in section 1,  in this paper,  we shall consider a simple case by assuming $\Omega=B_R(0)$ - a ball centered at origin with radius curvature is given by $\frac{1}{R}$. We find that the first two terms (zeroth and first order terms) of the (point-wise) asymptotic expansion of $u_\e(x)$ with respect to $\eps$ is adequate to help us find the role of the curvature on the boundary-layer structure (profile).

	To state our last results,  we introduce some notations. We denote by $\omega_N$ the volume of $B_R(0) \subset \mathbb{R}^N$ and $\alpha(N)=\frac{\pi^{\frac{N}{2}}}{\Gamma(\frac{N}{2}+1)}$ the volume of unit ball in $\mathbb{R}^N$, where $\omega_N=\alpha(N) R^N$. For convenience, we define
	\begin{align}\label{fts}
	f(s):=se^s\,\,\mathrm{and}\,\,F(s):=\int_0^sf(\tau)\mathrm{d}\tau=(s-1)e^s+1\geq0,\,\,\mathrm{for}\,\,s\geq0.
	\end{align}
	Then by the uniqueness (see Theorem~\ref{th1.1}) and {the classical moving plane method \cite{Ni}}, $\ueps(x)=\bigu(|x|)=\bigu(r)$ is radially symmetric in $B_R(0)$, where $\bigu$ uniquely solves
	\begin{align}
	&\eps^2\left(\bigu''+\frac{N-1}{r}\bigu'\right)={\meps}f(\bigu),\,\, r\in (0,R),\label{eq3}\\
	&{\meps}:={\meps}(\bigu)=\frac{m}{N\alpha(N)}\left(\int_0^Re^{\bigu(s)}s^{N-1}\dss\right)^{-1},\label{coef}\\
	&\bigu'(0)=0,\,\,\bigu(R)=u_0,\label{bd3}
	\end{align}
	where we remark that $N\alpha(N)$ is the surface area of the unit sphere $\partial{B}_1(0)$.
	
	Next we shall investigate how the boundary curvature will influence the boundary layer profile of \eqref{eq3}-\eqref{bd3} near the boundary from two different angles. The first one is to see how the slope of boundary layer profile at the boundary $r=R$ changes with respect to the boundary curvature $1/R$. The second one is for a given level set such that $\bigu(r_\eps)=c$, how the distance from boundary to the point $r_\eps$ varies with respect to $R$. To be more precisely,  for $R>0$ and $c\in(0,u_0)$, we define
	\begin{align}\label{gamma-thin}
	r_{\eps}(R,c):=\bigu^{-1}(c)\,\,\mathrm{and}\,\,
	\Gamma_{\eps}(R,c):=\{r\in[0,R]:\bigu(r)\in[c,u_0]\}=[r_{\eps}(R,c),R]
	\end{align}
	as functions of $R$ and $c$, where
	$\Gamma_{\eps}(R,c)$ is {a closed} set with width $R-r_{\eps}(R,c)=O(\eps)$. Denote
	\begin{align}\label{c-0}
	\hnu=-\sqrt{2F(u_0)}+\int_0^{u_0}\sqrt{\frac{F(t)}{2}}\,\mathrm{d}t.
	\end{align}
	Let $\Psi$ denote the unique positive solution of the ODE problem
	\begin{align}\label{phi-eq}
	\begin{cases}
	-\Psi'(\xi)=\sqrt{2m F(\Psi(\xi))},\,\, \ \xi>0,\\
	\Psi(0)=u_0>0,\,\,\Psi(\infty)=0.
	\end{cases}
	\end{align}
	
	Then our results on the refined asymptotic boundary layer profile in $\eps$ are given in the following theorem where we present a sharp pointwise asymptotic profile for  $\psi_\eps$ within the boundary layer and for the slope of $\psi_\eps$ at $r=R$ up to the first-order term of expansion in $\eps$, as well as the monotone property of the boundary layer thickness with respect to the boundary curvature.
	\begin{theorem}\label{mainthm}
		Let $m$ and $u_0$ be given positive constants and let $F$ and $\hnu$ be defined in \eqref{fts} and \eqref{c-0}, respectively. Then for any $\eps>0$, the solution $\bigu$ is positive and strictly increasing in $[0,R]$. Moreover, for any $0<\eps \ll 1$, we have the following results concerning the asymptotic expansion of $\psi_\eps$ with respect to $\eps$.
		\begin{itemize}
			\item[(i)]Let $r_{\eps}:=r_{\eps}(d_0)=R-d_0\eps\in(0,R]$  be a point with the distance $d_0\eps$ to the boundary, where the constant $d_0\geq0$ is independent of $\eps$. Then we have
			\begin{align}\label{asys-u}
			\bigu(r_{\eps}(d_0))
			=\Psi^R(d_0)-&\frac{\eps}{R}\sqrt{\frac{F(\Psi^R(d_0))}{2}}\notag\\[-0.7em]
			&\\[-0.7em]
			&\times\left({d_0N}{\hnu}-R^{N/2}\sqrt{\frac{\alpha(N)}{m}}(N-1)\jnu+o_{\eps}(1)\right)\notag
			\end{align}
			where $\Psi^R(d_0)=\Psi(\frac{d_0}{\sqrt{\alpha(N)R^{N/2}}})$ and
			\begin{align}
			{\jnu}=&\int_{\Psi^R(d_0)}^{u_0}\left(\frac{1}{F(s)}\int_0^{s}\sqrt{\frac{{F(t)}}{{F(s)}}}\,\mathrm{d}t\right)\mathrm{d}s.\label{wh-0717}
			\end{align}
			
			\item[(ii)] The slope of the boundary layer profile at the boundary has the asymptotic expansion as
			\begin{align}\label{0820-slope}
			\bigu'(R)=\frac{1}{\eps R^{N/2}}\sqrt{\frac{2m F(u_0)}{\alpha(N)}}+\frac{1}{R}\left(N\sqrt{\frac{F(u_0)}{2}}{\hnu}-(N-1)\int_0^{u_0}\sqrt{\frac{F(t)}{F(u_0)}}\,\mathrm{d}t\right)+o_{\eps}(1).
			\end{align}
			\item[(iii)] Let $r_{\eps}(R,c)$ be defined in \eqref{gamma-thin}. Then for each $c\in(0,u_0)$, we have
			\begin{equation}
			\label{good-asy}
			\begin{aligned}
			R-r_{\eps}(R,c)=~&\frac{\eps^2}{2}\alpha(N)R^{N-1}\left[{-\frac{N}{\sqrt{m}}\Psi^{-1}(c)\hnu}
			+\frac{N-1}{m}\int_c^{u_0}\left(\frac{1}{F(s)}\int_0^s\sqrt{\frac{F(t)}{F(s)}}dt\right)ds+o_{\eps}(1)\right]
			\\~&+\sqrt{\alpha(N)}R^{N/2}\Psi^{-1}(c)\eps.
			\end{aligned}
			\end{equation}
			In particular, for any $R_0>0$, there exists a positive constant $\delta_{N,R_0,c}$ depending mainly on $N$, $R_0$ and $c$ such that for each $\eps\in(0,\delta_{N,R_0,c})$, $R-r_{\eps}(R,c)$ is strictly increasing with respect to $R$ in \textcolor[rgb]{0,0,0}{$(0,R_0]$}.
		\end{itemize}
		
	\end{theorem}
	
	\begin{remark}
		The result of Theorem \ref{mainthm}-(i) implies that the slope of boundary layer profile near the boundary increases with respect to the boundary curvature (i.e. decrease with respect to $R$).  The result of Theorem \ref{mainthm}-(ii) implies that the boundary-layer thickness  decreases with respect to the boundary curvature (i.e. increases with respect to $R$). An illustration of curvature effect on the boundary layer profile is shown in Fig.\ref{fig2}.
	\end{remark}

	\begin{figure}[h]
		\includegraphics[width=10cm]{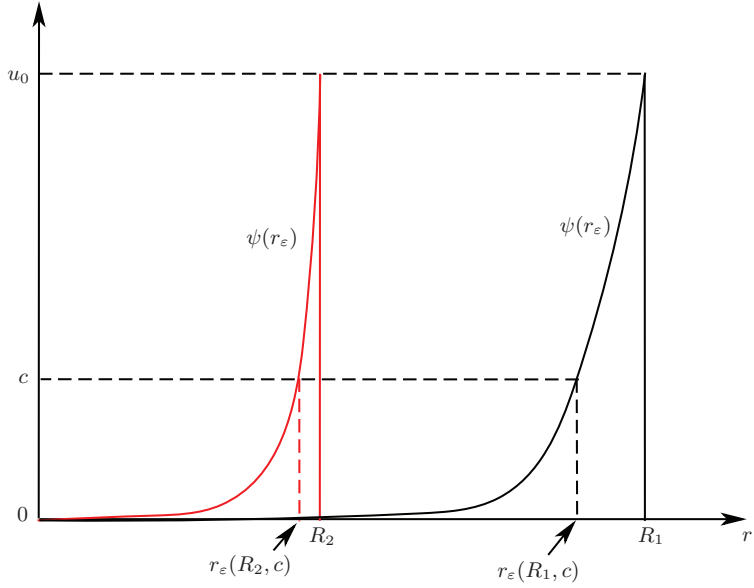}
		\caption{Schematic of the curvature effect on boundary layer profiles: layer steepness and thickness.}
		\label{fig2}
	\end{figure}

	\section{Proof of Theorems \ref{th1.1}}
	In this section, we shall prove Theorems \ref{th1.1}.
	
	\begin{proof}[Proof of Theorem \ref{th1.1}-existence  of  \eqref{1.model}]
		We start the proof by considering the following auxiliary  problem
		\begin{equation}
		\label{2.norm}
		\begin{cases}
		\e^2\Delta u_{\lambda}=\lambda mu_{\lambda}e^{u_{\lambda}}\quad &\mathrm{in}~\Omega,\\
		u_{\lambda}=u_0~&\mathrm{on}~\partial\Omega,
		\end{cases}
		\end{equation}
		where $\lambda$ is an arbitrary positive constant. Since $u_0>0$, by maximal principle we have
		$$u_0\geq u_{\lambda}>0\quad \mathrm{in}~\Omega.$$
		Then it is not difficult to see that $u\equiv u_0$ is a super-solution of \eqref{2.norm}, while $u\equiv 0$ provides a sub-solution. Therefore, following the standard monotone iteration scheme and the fact that $f(x)=xe^x$ is increasing for $x$ positive, we immediately know that \eqref{2.norm} has a unique classical solution $u_{\lambda}$ verifying
		$$0<u_{\lambda}\leq u_0.$$
		
		Now we claim that there exists $\lambda_{m}>0$ such that
		\begin{equation}
		\label{2.const}
		\lambda_m\int_{\Omega}e^{u_{\lambda_m}}{\dx}=1.
		\end{equation}
		We postpone the proof of \eqref{2.const} in Lemma \ref{le2.1}. Using this claim we can easily see that $u=\u$ is a solution of \eqref{1.model}.
	\end{proof}
	
	In order to prove the claim \eqref{2.const}, we give the following lemma.
	\begin{lemma}
		\label{le2.1}
		Let $\lambda_1\geq \lambda_2>0$ and $u_{\lambda_i},~i=1,2$ be the solutions of \eqref{2.norm} with $\lambda=\lambda_i,~i=1,2$ respectively. Then
		\begin{equation}
		\label{2.comp}
		0\leq u_{\lambda_2}-u_{\lambda_1}\leq\left(\frac{\lambda_1}{\lambda_2}-1\right)
		\left(1+\frac{\lambda_1}{\lambda_2}e^{\frac{\lambda_1}{\lambda_2}u_0}\right)u_0.
		\end{equation}
		Moreover, there exists a constant $\lambda_m$ such that
		$$\lambda_m\int_{\Omega}e^{\u}{\dx}=1.$$
	\end{lemma}
	
	\begin{proof}
		The left hand side inequality follows from the standard comparison theorem directly. We only prove the inequality for the right hand side. Due to the fact $\lambda_1\geq \lambda_2>0$ and $u_{\lambda_1}>0$, one may check that
		\begin{equation}
		\label{2.est}
		\begin{aligned}
		\e^2 \Delta\left(\frac{\lambda_1}{\lambda_2}u_{\lambda_1}-u_{\lambda_2}\right)
		&\leq \lambda_2m\left(\frac{\lambda_1}{\lambda_2}u_{\lambda_1}e^{\frac{\lambda_1}{\lambda_2}u_{\lambda_1}}
		-u_{\lambda_2}e^{u_{\lambda_2}}\right)
		+\lambda_1\left(\frac{\lambda_1}{\lambda_2}-1\right)mu_{\lambda_1}e^{\frac{\lambda_1}{\lambda_2}u_{\lambda_1}}\\
		&\leq \lambda_2mF(u_{\lambda_1},u_{\lambda_2})\left(\frac{\lambda_1}{\lambda_2}u_{\lambda_1}-u_{\lambda_2}\right)
		+\lambda_1\left(\frac{\lambda_1}{\lambda_2}-1\right)mu_0e^{\frac{\lambda_1}{\lambda_2}u_0},
		\end{aligned}
		\end{equation}
		where
		\begin{equation*}
		F(a,b):=
		\begin{cases}
		\frac{ae^a-be^b}{a-b},\quad &\mathrm{if}~a\neq b,\\
		(a+1)e^a,&\mathrm{if}~a=b.
		\end{cases}
		\end{equation*}
		From the fact
		\begin{equation}
		\label{2.fact}
		0<u_{\lambda_1}\leq u_{\lambda_2}\leq u_0,
		\end{equation}
		we have
		\begin{equation}
		\label{2.bound}
		1< F(u_{\lambda_1},u_{\lambda_2})\leq (1+u_0)e^{u_0}.
		\end{equation}
		As a consequence of \eqref{2.est} and \eqref{2.bound}, we have
		\begin{equation*}
		\frac{\lambda_1}{\lambda_2}u_{\lambda_1}-u_{\lambda_2}\geq-\frac{\lambda_1}{\lambda_2}
		\left(\frac{\lambda_1}{\lambda_2}-1\right)u_0e^{\frac{\lambda_1}{\lambda_2}u_0}.
		\end{equation*}
		Along with \eqref{2.fact}, we get
		\begin{equation*}
		\begin{aligned}
		u_{\lambda_2}-u_{\lambda_1}=&~
		u_{\lambda_2}-\frac{\lambda_1}{\lambda_2}u_{\lambda_1}
		+\left(\frac{\lambda_1}{\lambda_2}u_{\lambda_1}-u_{\lambda_1}\right)
		\leq u_{\lambda_2}-\frac{\lambda_1}{\lambda_2}u_{\lambda_1}+\left(\frac{\lambda_1}{\lambda_2}-1\right)u_0\\
		\leq &\left(\frac{\lambda_1}{\lambda_2}-1\right)
		\left(1+\frac{\lambda_1}{\lambda_2}e^{\frac{\lambda_1}{\lambda_2}u_0}\right)u_0.
		\end{aligned}
		\end{equation*}
		Thus, we prove the right hand side of \eqref{2.comp}. It implies that the continuity of $u_{\lambda}$ with respect to $\lambda$. On the other hand, we have
		\begin{equation*}
		\lim_{\lambda\to 0^+}\lambda\int_{\Omega}e^{u_\lambda}{\dx}=0\quad \mathrm{and}\quad
		\lim_{\lambda\to \infty}\lambda\int_{\Omega}e^{u_\lambda}{\dx}=\infty.
		\end{equation*}
		Then we can find $\lambda_m$ such that $\int_{\Omega}\lambda_me^{u_{\lambda_m}}{\dx}=1$ and it completes the proof of Lemma \ref{le2.1}.
	\end{proof}

	\begin{proof}[Proof of Theorem \ref{th1.1} -- uniqueness of \eqref{1.model}]
		Suppose the uniqueness is not true, then there are two distinct solutions $v_1,v_2$. We shall prove $v_1\equiv v_2$ by contradiction and divide our argument into two steps.
		
		\textbf{Step 1.} We prove that either $v_1\geq v_2$ or $v_1\leq v_2$. Without loss of generality, we may assume $\int_{\Omega} e^{v_1}{\dx}\geq \int_{\Omega} e^{v_2}dx.$  Under this assumption, we claim $v_1\geq v_2$. Supposing this is false, then there exists a point $p\in\Omega,$ such that
		\begin{equation*}
		(v_1-v_2)\mid_p=\min_{\Omega}(v_1-v_2)<0.
		\end{equation*}
		As a consequence, we have
		\begin{equation*}
		\left(\dfrac{v_1e^{v_1}}{\int_{\Omega}e^{v_1}\dx}-\frac{v_2e^{v_2}}{\int_{\Omega}e^{v_2}{\dx}}\right)\bigg|_p<0.
		\end{equation*}
		Then
		\begin{equation*}
		\left[\e^2\Delta(v_1-v_2)-m\left(\dfrac{v_1e^{v_1}}{\int_{\Omega}e^{v_1}{\dx}}
		-\dfrac{v_2e^{v_2}}{\int_{\Omega}e^{v_2}{\dx}}\right)\right]\bigg|_p>0.
		\end{equation*}
		Contradiction arises. Thus, the claim holds. As a result, we get that for any two solutions $v_1,v_2$,  either $v_1\geq v_2$ or $v_1\leq v_2$.
		
		\textbf{Step 2.} Next, we prove that if $v_1\geq v_2$, then $v_1=v_2$. We set $w=v_1-v_2$, suppose $w\neq0$ and
		\begin{equation*}
		w(p_0)=\max_{\Omega}w>0.
		\end{equation*}
		Then
		\begin{equation*}
		\dfrac{e^{v_1(p_0)}}{e^{v_2(p_0)}}\geq\dfrac{e^{v_1}}{e^{v_2}}\quad \mathrm{in}\quad \Omega.
		\end{equation*}
		It implies that
		\begin{equation*}
		\dfrac{e^{v_1(p_0)}}{e^{v_2(p_0)}}\geq\dfrac{\int_{\Omega}e^{v_1}{\dx}}{\int_{\Omega}e^{v_2}{\dx}},
		\end{equation*}
		and
		\begin{equation*}
		\left(\frac{v_1e^{v_1}}{\int_{\Omega}e^{v_1}{\dx}}-\frac{v_2e^{v_2}}{\int_{\Omega}e^{v_2}{\dx}}\right)\Big|_{p_0}>0.
		\end{equation*}
		Therefore,
		\begin{equation*}
		\e^2\Delta(v_1-v_2)\Big|_{p_0}=m\left(\frac{v_1e^{v_1}}{\int_{\Omega}e^{v_1}{\dx}}
		-\frac{v_2e^{v_2}}{\int_{\Omega}e^{v_2}{\dx}}\right)\Big|_{p_0}>0,
		\end{equation*}
		which contradicts to the choice of the point $p_0$, thus $v_1(p_0)=v_2(p_0)$ and $v_1\equiv v_2.$
	\end{proof}
	
	\begin{proof}[Proof of Theorem \ref{th1.1} -- existence and uniqueness of \eqref{1.steady} with non-degeneracy.]
		Since the problem \eqref{1.steady} is equivalent to \eqref{1.model}-\eqref{vu}, the existence and uniqueness of solutions to  \eqref{1.steady} follow directly from the results for \eqref{1.model}. Now it remains to show the solution is non-degenerate. We denote the solution of \eqref{1.steady} by $(u,v)$ and consider the linearized problem of \eqref{1.steady} at $(u,v):$
		\begin{equation}\label{2.linear}
		\begin{cases}
		\nabla\cdot(\nabla \phi-v\nabla \psi-\nabla u\phi)=0, \ & \mathrm{in}~\Omega,\\
		\Delta\psi-\phi u-\psi v=0, \ &\mathrm{in}~\Omega,\\
		\partial_{\nu}\phi-v\partial_{\nu}\psi-\partial_{\nu}u\phi=0,~\psi=0, \ & \mathrm{on}~\partial\Omega,\\
		\int_{\Omega}\phi {\dx}=0.
		\end{cases}
		\end{equation}
		We shall prove that \eqref{2.linear} only admits the trivial solution. We notice that the first equation in \eqref{2.linear} can be written as
		\begin{equation*}
		\nabla\cdot\left(v\nabla\left(\frac{\phi}{v}-\psi\right)\right)=0,
		\end{equation*}
		where we used the fact $\nabla u=\frac{\nabla v}{v}$. Testing the above equation by $\frac{\phi}{v}-\psi,$ then an integration by parts together with the boundary condition shows that any solution of \eqref{2.linear} verifies the equation
		\begin{equation*}
		\int_{\Omega}v\left|\nabla\left(\frac{\phi}{v}-\psi\right)\right|^2{\dx}=0,
		\end{equation*}
		which implies that
		\begin{equation}
		\label{2.relation}
		\phi=(\psi+E)v\quad \mbox{for some constant }E.
		\end{equation}
		Since $\int_{\Omega}\phi {\dx}=0$, we get from \eqref{2.relation} that if $\psi$ is not a constant, then
		\begin{equation}
		\label{2.inequality}
		\max_{\Omega}\psi+E>0\quad\mbox{and}\quad\min_{\Omega}\psi+E<0.
		\end{equation}
		Substituting \eqref{2.relation} into the second equation of \eqref{2.linear}, we have
		\begin{equation}
		\label{2.reduce-1}
		\begin{cases}
		\Delta\psi-v\psi-uv(\psi+E)=0\quad&\mathrm{in}~\Omega,\\
		\psi=0\quad&\mathrm{on}~\partial\Omega.
		\end{cases}
		\end{equation}
		We claim that equation \eqref{2.reduce-1} only admits the trivial solution. Supposing that it is false, without loss of generality, we can assume that $\psi(p)=\max\limits_{\Omega}\psi>0.$ As a consequence,
		\begin{equation*}
		\Delta\psi(p)-v\psi(p)-uv(\psi(p)+E)<0,
		\end{equation*}
		where we have used \eqref{2.inequality}, and contradiction arises. Thus $\psi\equiv0$, which further implies that $\phi\equiv0$ from the second equation of \eqref{2.linear}. This means that any solution of \eqref{1.steady} is non-degenerate.
	\end{proof}

	\medskip
	\section{Proof of Theorems \ref{th1.2} and Corollary \ref{th1.3}}
	
	In order to prove Theorem \ref{th1.2}, we consider the following equation
	\begin{equation}
	\label{3.1}
	\begin{cases}
	\e^2\Delta V=d^2V \quad &\mathrm{in}~\Omega,\\
	V=1&\mathrm{on}~\partial\Omega,
	\end{cases}
	\end{equation}
	where $d$ is a positive constant independent of $\e$. Set $W_{\e}=-\e\log V$. Then by the arguments in \cite[Lemma 2.1]{gw}, we have
	\begin{equation*}
	\begin{cases}
	W_{\e}(x)=d\mathrm{dist}(x,\partial\Omega)+O(\e)~&\mathrm{in}~\Omega,\\
	\frac{\partial{W}_{\e}}{\partial\nu}=-d+O(\e)~&\mathrm{on}~\partial\Omega.
	\end{cases}
	\end{equation*}
	and
	\begin{equation}
	\label{3.2-asy}
	|V(x)|\leq Ce^{-d\frac{\mathrm{dist}(x,\partial\Omega)}{\e}}\quad\mathrm{in}\quad\Omega.
	\end{equation}
	As a consequence of \eqref{3.2-asy}, we have for any compact subset $K$ of $\Omega$, there exists a positive constant $\e_0$ such that
	\begin{equation}
	\label{3.3-asy}
	\max_K|V|\leq C(K)e^{-\frac{M(K)}{\e}}\quad\mathrm{for}\quad 0<\e<\e_0,
	\end{equation}
	where $C(K)$ and $M(K)$ are some generic constants depending on $K$ only.
	
	From \eqref{3.3-asy}, it is easy to see that for any fixed compact subset $K$ of $\Omega$, we could obtain that $V$ goes to $0$ as $\e$ tends to $0.$ To capture the behavior of $V$ near $\partial\Omega,$ we introduce the Fermi coordinates for any $x\in\Omega_\delta^c$, that is
	$$X:(y,z)\in\partial\Omega\times\mathbb{R}^+\longmapsto x=X(y,z)=y+z\nu(y)\in\Omega_\delta^c,$$
	where $\nu$ is the unit normal vector on $\partial\Omega$, and $\Omega_\delta^c$ denotes the following open set
	\begin{equation}
	\label{3.com-set}
	\Omega_{\delta}^c=\{x\in\Omega\mid 0<\mathrm{dist}(x,\partial\Omega)<\delta\}.
	\end{equation}
	There is a number $\delta_0>0$ such that for any $\delta\in(0,\delta_0)$, the map $X$ is {from $\Omega_{\delta}^c$ to a subset of $\mathcal{O}$, where}
	\begin{equation*}
	\mathcal{O}=\{(y,z)\in\partial\Omega\times(0,2\delta)\}.
	\end{equation*}
	It follows that $X$ is actually a diffeomorphism onto its image $\mathcal{N}=X(\mathcal{O})$. We refer the readers to \cite[Remark 8.1]{dkw} for the proof on the existence of $\delta_0$. For any fixed $z$, we set
	\begin{equation*}
	\Gamma_z(y)=\{p\in\Omega\mid p=y+z\nu(y)\}.
	\end{equation*}
	It is not difficult to see that the distance between any point of $\Gamma_z(y)$ and $\partial\Omega$ is $|z|$. Under the Fermi coordinate, we have
	\begin{equation}
	\label{3.1-fc}
	\Delta=\partial_z^2-H_{\Gamma_z(y)}\partial_z+\Delta_{\Gamma_z(y)},
	\end{equation}
	where $H_{\Gamma_z(y)}$ is the mean curvature at the point in $\Gamma_z(y)$ and $\Delta_{\Gamma_z(y)}$ stands for the Beltrami-Laplacian on $\Gamma_z(y)$. We shall provide the proof of \eqref{3.1-fc} in the Appendix, see Lemma \ref{lea.1}.
	
	By making use of \eqref{3.3-asy}, we can obtain the following result on the behavior of $V$ near $\partial\Omega$.
	\begin{lemma}
		\label{le3.1}
		Let $\Omega$ be a smooth domain in $\mathbb{R}^N (N\geq 1)$ and $V_{\e}\in C^{2,\alpha}(\Omega)\cap C^0(\overline{\Omega})$ be the unique solution of \eqref{3.1}. There exists a positive constant $\e_0$ such that for any $\e\in(0,\e_0)$, it holds that
		\begin{equation}
		\label{3.2}
		b_1e^{-b_2\frac{dist(x,\partial\Omega)}{\e}}\leq V_{\e}(x)\leq b_3e^{-b_4\frac{dist(x,\partial\Omega)}{\e}}\quad\mathrm{in}\quad \Omega^c_{\delta},
		\end{equation}
		where $\delta\in(0,\min\{\frac12,\delta_0\})$, and $b_1,b_2,b_3$, $b_4$ are some generic positive constants independent of $\e.$
	\end{lemma}

	\begin{proof}
		For convenience we set $d=1$. When $n=1$, without loss of generality we can assume that $\Omega=[-1,1]\subset\mathbb{R}$, it is easy to see that the solution admits the following representation
		\begin{equation*}
		V_{\e}(x)=\dfrac{1}{1+e^{-\frac{2}{\e}}}\left(e^{-\frac{(x+1)}{\e}}+e^{\frac{(x-1)}{\e}}\right)
		\quad\mathrm{for}\quad x\in[-1,1].
		\end{equation*}
		Hence, Lemma \ref{le3.1} immediately follows. Particularly, in this case we can choose
		$$b_1=\dfrac{1}{1+e^{-\frac{2}{\e}}},~b_2=1,~b_3=\dfrac{2}{1+e^{-\frac{2}{\e}}},~b_4=1.$$
		
		Now let us give the proof for $n\geq2$. Without loss of generality, we may assume that $\Omega$ is a simply connected domain for simplicity, the case for multiply-connected domain can be proved similarly. Since $\Omega$ is simply connected, $\partial\Omega$ is a smooth connected manifold of dimension $n-1.$ Let $\Omega_{2\delta}^c$ be defined in \eqref{3.com-set}. We set $u_{\delta}$ by
		\begin{equation*}
		u_{\delta}(x)=(2\delta-\mathrm{dist}(x,\partial\Omega))e^{-d_1\frac{\mathrm{dist}(x,\partial\Omega)}{\e}}
		\end{equation*}
		with $d_1$ to be determined later. It is easy to see that
		\begin{equation}
		\label{3.1-4}
		u_{\delta}(x)=
		\begin{cases}
		2\delta~&\mathrm{on}\quad\partial\Omega,\\
		0~&\mathrm{on}\quad \partial\Omega_{2\delta}^c\setminus\partial\Omega.
		\end{cases}
		\end{equation}
		A straightforward computation with \eqref{3.1-fc} gives
		\begin{equation*}
		\begin{aligned}
		\e^2\Delta u_{\delta}-u_{\delta}=~&(\e^2\partial_z^2-\e^2H_{\Gamma_z(y)}\partial _z+\e^2\Delta_{\Gamma_z(y)}-1)(2\delta-z)e^{-d_1\frac{z}{\e}}\\
		=~&(d_1^2+\e H_{\Gamma_z(y)}d_1-1)u_\delta+(2\e d_1+\e^2H_{\Gamma_z(y)})e^{-d_1\frac{z}{\e}}.
		\end{aligned}
		\end{equation*}
		Choosing $d_1>1$ and $\e$ sufficiently small, then we have
		\begin{equation*}
		\e^2\Delta u_{\delta}-u_{\delta}\geq0\quad\mathrm{in}~{\Omega_{2\delta}^c}.
		\end{equation*}
		Taking $\delta$ sufficiently small when necessary, together with \eqref{3.1-4} and the classical comparison argument, we have
		\begin{equation*}
		V_{\e}\geq u_{\delta}(x)\quad\mathrm{in}~\overline{\Omega_{2\delta}^c},
		\end{equation*}
		which implies
		\begin{equation}
		\label{3.1-7}
		V_{\e}\geq \delta e^{-d_1\frac{\mathrm{dist}(x,\partial\Omega)}{\e}}\quad\mathrm{in}~\overline{\Omega_{\delta}^c}.
		\end{equation}
		For the upper bound for $V_{\e}$ in $\Omega_\delta^c$,  by \eqref{3.2-asy}, we have
		\begin{equation}
		\label{3.1-up}
		V_{\e}\leq Ce^{-\frac{\mathrm{dist}(x,\partial\Omega)}{\e}}\quad\mathrm{in}\quad \overline{\Omega_\delta^c}.
		\end{equation}
		From \eqref{3.1-7} and \eqref{3.1-up}, we get
		\begin{equation*}
		\delta e^{-d_1\frac{\mathrm{dist}(x,\partial\Omega)}{\e}}\leq V_{\e}(x)\leq Ce^{-\frac{\mathrm{dist}(x,\partial\Omega)}{\e}}\quad\mathrm{in}\quad \Omega_{\delta}^c.
		\end{equation*}
		This is equivalent to \eqref{3.2} with
		$$b_1=\delta,~b_2=d_1,~b_3=C,~b_4=1.$$
		Thus, we finish the proof.
	\end{proof}
	\medskip

	\begin{proof}[Proof of Theorem \ref{th1.2}]
		For \eqref{1.model}, by maximal principle, we have
		\begin{equation*}
		0<u_\e<u_0.
		\end{equation*}
		Then it is easy to check that
		\begin{equation*}
		0<\frac{1}{\int_{\Omega}e^{u_0}{\dx}}\leq \frac{e^{u_\e}}{\int_{\Omega}e^{u_\e}{\dx}}\leq \frac{e^{u_0}}{|\Omega|}.
		\end{equation*}
		We set
		\begin{equation}
		\label{3.2-3}
		L_1:=\left(\frac{1}{\int_{\Omega}e^{u_0}{\dx}}\right)^{\frac12}\quad \mathrm{and}\quad
		L_2:=\left(\frac{e^{u_0}}{|\Omega|}\right)^{\frac12}.
		\end{equation}
		Let $u_{\e,L_1}$ and $u_{\e,L_2}$ be the solution of \eqref{3.1} with the righthand side replaced by $L_1^2V$ and $L_2^2V$ respectively. Then following the comparison argument, we can get
		\begin{equation*}
		u_{\e,L_2}\leq u_\e \leq u_{\e,L_1}.
		\end{equation*}
		As a consequence of Lemma \ref{le3.1}, we can find four constants $b_5,b_6,b_7,b_8$ which are independent of $\e,$ such that
		\begin{equation}
		\label{3.2-5}
		b_5e^{-b_6\frac{\mathrm{dist}(x,\partial\Omega)}{\e}}\leq u_\e(x)\leq b_7e^{-b_8\frac{\mathrm{dist}(x,\partial\Omega)}{\e}}\quad\mathrm{in}\quad \Omega^c_{\delta}.
		\end{equation}
		While in $\overline{\Omega_\delta}$, by equation \eqref{3.3-asy} we can find two positive constants $C(\overline{\Omega_\delta})$ and $M(\overline{\Omega_\delta})$ such that
		\begin{equation}
		\label{3.2-6}
		\max_{\overline{\Omega_\delta}}u_\e(x)\leq C(\overline{\Omega_\delta})e^{-\frac{M(\overline{\Omega_\delta})}{\e}}\quad\mathrm{in}\quad \overline{\Omega_\delta}.
		\end{equation}
		Then \eqref{1.th2-1} follows by \eqref{3.2-5} and \eqref{3.2-6}.
		\medskip
		
		In the following, we shall prove \eqref{1.th2-2}. We first prove that
		\begin{equation}
		\label{3.3-0}
		|\Omega|<\int_{\Omega}e^{u_\e}{\dx}\leq |\Omega|+C\e,
		\end{equation}
		for some positive constant $C$. The left hand side of \eqref{3.3-0} is obvious since $u_\e>0$ in $\Omega$. For the right hand side inequality, by \eqref{3.2-asy} and $u_\e<u_0$ in $\Omega$, we have
		\begin{equation}
		\label{3.3-01}
		\int_{\Omega}e^{u_\e}{\dx}\leq \int_{\Omega}1{\dx}+C\int_{\Omega}u_\e {\dx}
		\end{equation}
		for some positive constant $C.$ For the second term on the right hand side of \eqref{3.3-01}, we have
		\begin{equation}
		\label{3.3-02}
		\int_{\Omega}u_\e {\dx}=\int_{\Omega^c_\delta}u_\e {\dx}+\int_{\Omega_\delta}u_\e {\dx}\leq \int_{\Omega^c_\delta}u_\e {\dx}+O(\e),
		\end{equation}
		where we used \eqref{3.2-asy} to control the second term. While for the first one, we have
		\begin{equation*}
		\begin{aligned}
		\int_{\Omega^c_\delta}u_\e {\dx}&\leq C\int_0^\delta\int_{\Gamma_z(y)}e^{-\frac{z}{\e}}\mathrm{d}z\mathrm{d}y
		\leq C\max_{z\in(0,\delta)}|\Gamma_z(y)|\int_0^\delta e^{-\frac{z}{\e}}\mathrm{d}z\\
		&=C\e\max_{z\in(0,\delta)}|\Gamma_z(y)|\int_0^\frac{\delta}{\e}e^{-s}\mathrm{d}s\leq C\e\max_{z}|\Gamma_z(y)|.
		\end{aligned}
		\end{equation*}
		By choosing $\delta$ small, we have $\max_{z\in(0,\delta)}|\Gamma_z(y)|=|\partial\Omega|+o_\delta(1)$, where $o_\delta(1)\xrightarrow{\delta\to0}0$. Hence, we have
		\begin{equation*}
		\int_{\Omega^c_\delta}u_\e {\dx}=O(\e),
		\end{equation*}
		which together with \eqref{3.3-02} gives \eqref{3.3-0}. Using \eqref{3.3-0}, it is not difficult to see that
		\begin{equation*}
		\frac{1}{\int_{\Omega}e^{u_\e}{\dx}}=\frac{1}{|\Omega|+\int_{\Omega}(e^{u_\e}-1){\dx}}=\frac{1}{|\Omega|}-C_{u_\e},
		\end{equation*}
		where
		\begin{equation*}
		0<C_{u_\e}<C_0\e.
		\end{equation*}
		for some constant $C_0$. We decompose
		\begin{equation*}
		u_\e(x)=U_{\e}(x/\e)+\phi_\e(x),
		\end{equation*}
		where $U_{\e}$ is the solution of \eqref{1.def-u1} and
		\begin{equation}
		\label{3.3-4}
		\begin{cases}
		\e^2\Delta \phi_{\e}=-mC_{u_\e}u_\e e^{u_\e}+\frac{m}{|\Omega|}\left(e^{u_\e}u_\e-e^{U_{\e}}U_{\e}\right)~&\mbox{in}~\Omega,\\
		\phi_{\e}=0~&\mbox{on}~\partial\Omega.
		\end{cases}
		\end{equation}
		It is easy to see that $U_{\e}>0$ in $\Omega.$ We write the first equation in \eqref{3.3-4} as
		\begin{equation*}
		\e^2\Delta \phi_{\e}-\frac{m}{|\Omega|}\left(e^{U_{\e}+\phi_{\e}}(U_{\e}+\phi_{\e})
		-e^{U_{\e}}U_{\e}\right)
		=-mC_{u_\e}u_\e e^{u_\e}.
		\end{equation*}
		Concerning the above equation, by the fact that the function $f(x)=xe^x$ is an increasing function for $x>0$ and the right hand side is negative, we get
		$$\phi_{\e}>0\quad\mathrm{in}\quad \Omega$$
		by maximal principle. Assuming $\phi_{\e}(p)=\max_{\Omega}\phi_{\e}$, by Mean Value Theorem, we have
		\begin{equation*}
		\e^2\Delta \phi_{\e}
		-\frac{m}{|\Omega|}\left(e^{U_{\e}+\theta\phi_{\e}}(U_{\e}+\theta\phi_{\e})
		+e^{U_{\e}+\theta\phi_{\e}}\right)\phi_{\e}=-mC_{u_\e}u^{\e}
		\end{equation*}
		for some $\theta\in(0,1)$. Together with the fact that $\e^2\Delta\phi_{\e}(p)<0$, we directly obtain that
		\begin{equation*}
		\phi_{\e}(p)=O(\e),
		\end{equation*}
		which proves that $\|\phi_{\e}\|_{L^\infty(\Omega)}=O(\e)$ and it finishes the proof.
	\end{proof}

	\begin{proof}[Proof of Corollary~\ref{th1.3}]
		We shall present the proof for three cases in Corollary~\ref{th1.3} separately.
		
		Case (1):  $\lim_{\e\to0}\frac{\ell_\e}{\e}=0$. Let $v^\e(y)=u_\e(\e y)$, then $v^\e(y)$ satisfies
		\begin{equation}
		\label{3.4-1}
		\Delta_yv^\e(y)=m\frac{v^{\e}(y)e^{v^{\e}(y)}}{\int_{\Omega^\e}e^{v^\e(y)}\mathrm{d}y}.
		\end{equation}
		Recall that, by maximal principle, we have
		$$\|v^{\e}(y)\|_{L^\infty(\Omega^\e)}=\|u_\e(x)\|_{L^\infty(\Omega)}\leq u_0.$$
		Following the standard elliptic estimate and the fact that the right-hand side of \eqref{3.4-1} is uniformly bounded, we get
		\begin{equation*}
		|v^{\e}(y)|_{L^\infty(\Omega^\e)}+|D_yv^{\e}(y)|_{L^\infty(\Omega^\e)}\leq C,
		\end{equation*}
		where $C$ is a universal constant and independent of $\e$. It implies that
		\begin{equation*}
		|D_xu_\e(x)|\leq C\e^{-1}.
		\end{equation*}
		Let $x_0$ be the boundary point such that
		\begin{equation*}
		|x_0-x_{in}|=\mathrm{dist}(x_{in},\partial\Omega).
		\end{equation*}
		We get that $|x_0-x_{in}|=o(\e)$ from $\lim_{\e\to0}\frac{l_{\e}}{\e}=0$, then
		\begin{equation*}
		|u_\e(x_0)-u_\e(x_{in})|\leq C|D_xu_\e||x_{in}-x_0|\leq C\e^{-1}|x_{in}-x_0|=o_{\e}(1),
		\end{equation*}
		which implies that $\lim_{\e\to0}u_\e(x_{in})=u_0$. This proves the conclusion of case (1) in Corollary~\ref{th1.3}.
		
		Case (2): $\lim_{\e\to0}\frac{\ell_\e}{\e}=L$. In this case, we first show that $\lim_{\e\to0}u_\e(x_{in})>0$. Indeed, by \eqref{3.2-5} and $\lim\limits_{\e\to0}l_{\e}/l=L$, we have
		$$\lim_{\e\to0}u_\e(x_{in})\geq b_5e^{-b_6L}>0.$$
		To show $\lim_{\e\to0}u_\e(x_{in})<u_0$, we claim that
		\begin{equation}
		\label{3.4-4}
		u_\e(x)\leq u_0e^{-b_9\frac{\mathrm{dist}(x,\partial\Omega)}{\e}}\quad\mathrm{in}\quad \Omega_\delta^c
		\end{equation}
		for some suitable positive constant $b_9,$ where $\Omega_\delta^c$ is defined in \eqref{3.com-set}. Let $L_1$ be defined in \eqref{3.2-3} and $u_{\e,b}$ be the solution of the following equation
		\begin{equation*}
		\begin{cases}
		\e^2\Delta V=L_1^2V\quad &\mathrm{in}~\Omega,\\
		V=u_0 &\mathrm{on}~\partial\Omega.
		\end{cases}
		\end{equation*}
		By maximum principle, we get that
		\begin{equation*}
		u_\e\leq u_{\e,b}.
		\end{equation*}
		Same as \eqref{3.2-asy}, we have
		\begin{equation*}
		|u_{\e,b}|\leq Ce^{-L_1\frac{\mathrm{dist}(x,\partial\Omega)}{\e}}\quad
		\mathrm{in}\quad\Omega.
		\end{equation*}
		Now we prove that
		\begin{equation*}
		u_{\e,b}(x)\leq u_0e^{-b_{10}\frac{\mathrm{dist}(x,\partial\Omega)}{\e}}\quad\mathrm{in}\quad \Omega_\delta^c
		\end{equation*}
		for some suitable positive constant $b_{10}$. We define $v_{\delta}$ by
		\begin{equation*}
		v_{\delta}=u_0e^{-\tau\frac{\mathrm{dist}(x,\partial\Omega)}{\e}}
		\end{equation*}
		with $\tau$ to be determined. On $\partial\Omega_\delta^c\setminus\partial\Omega$, we choose $\tau<L_1$ small enough such that
		\begin{equation*}
		u_0e^{-\tau\frac{\delta}{\e}}\geq Ce^{-L_1\frac{\delta}{\e}}.
		\end{equation*}
		Therefore
		\begin{equation}
		\label{3.4-11}
		u_{\e,b}\leq v_{\delta}\quad \mathrm{on}\quad \partial\Omega_\delta^c.
		\end{equation}
		By a direct computation, we have
		\begin{equation*}
		\e^2\Delta v_{\delta}-L_1^2v_{\delta}=-(L_1^2-\tau^2-\e \tau H_{\Gamma_z(y)})v_\delta.
		\end{equation*}
		For sufficiently small $\e$, we have
		$$\e^2\Delta v_\delta-L_1^2 v_\delta\leq 0.$$
		With \eqref{3.4-11} and the standard comparison argument, we get
		\begin{equation*}
		u_\e\leq u_{\e,b}\leq v_{\delta}\quad \mathrm{in}\quad \Omega_\delta^c.
		\end{equation*}
		Choosing $b_9=b_{10}=\tau$, we derive the claim \eqref{3.4-4}. As a result, we have
		\begin{equation*}
		\lim_{\e\to0}u_\e(x_{in})\leq u_0e^{-b_9L}<u_0.
		\end{equation*}
		Hence, we get the second conclusion.
		
		Case (3): $\lim_{\e\to0}\frac{\ell_\e}{\e}=+\infty$. The conclusion for this case is a direct consequence of \eqref{3.2-5}. Thus, we complete the proof.
	\end{proof}

	\section{Proof of Theorem~\ref{mainthm}}
	\subsection{Refined estimates of $\meps$}\label{sect=pro-m}
	We remark that \eqref{eq3}--\eqref{bd3} does not have a variational structure, and the nonlocal coefficient $\meps$ depends on the unknown solution $\bigu$. Hence, the variational approach and the standard method of matched asymptotic expansions~\cite{GJT2016,JPT2017} for singularly perturbed elliptic equations cannot be applied directly to our problem. On the other hand, as $\eps$ goes to zero, $\meps$ tends to a positive constant $\frac{m}{\omega_N}$. This enables us get the precise leading-order term of $\bigu$ near the boundary which encapsulate many useful properties of  $\bigu$ and hence motivates us to establish the precise leading order term of $\meps-\frac{m}{\omega_N}$ for small $\eps>0$. Based on a technical analysis developed in \cite[Theorem~4.1(III)]{L2019} and \cite[Lemma~4.1]{L2019-s} and Poho\v{z}aev-type identity for \eqref{eq3}--\eqref{bd3} (see Lemma~\ref{lem-2}), we gradually derive the zeroth and first order terms of $\meps$ (see Proposition~\ref{pro-m}) in the following.
	
	{By the arguments in section 4, we obtain that there exist positive constants $C_1$ and $M_1$ independent of $\e$, such that
		\begin{equation}
		\label{u-est1}
		0<\ueps(x)\leq{C_1}e^{-\frac{M_1}{\eps}\mathrm{dist}(x,\partial\Omega)}\quad\mathrm{for}\quad x\in\overline{\Omega},
		\end{equation}}
	which of course implies
	\begin{align}\label{u-est2}
	0<\bigu(r)\leq{C}_1e^{-\frac{M_1}{\eps}(R-r)},\,\,r\in[0,R].
	\end{align}
	Then by the Dominant Convergence Theorem, we have from \eqref{coef} that
	\begin{align}\label{mu-asy}
	\lim_{\eps\rightarrow0}\meps=\frac{m}{\omega_N}=\frac{m}{\alpha(N)R^N}:=\wmr.
	\end{align}
	With simple calculations,  we find the equation \eqref{eq3} can be transformed into an integro-ODEs
	\begin{align}\label{0710-id3}
	\frac{\eps^2}{2}\bigu'^2(r)+\eps^2\int_{\frac{R}{2}}^r\frac{N-1}{s}\bigu'^2(s)\,\dss={\meps}{F(\bigu(r))}+\mathtt{K}_{\eps},\,\,r\in[0,R],
	\end{align}
	where $\mathtt{K}_{\eps}$ is a constant depending on $\eps$. The equation \eqref{0710-id3} plays a crucial role in studying the asymptotic behavior of $\bigu$ near the boundary. To obtain the refined asymptotics of the nonlocal coefficient $\meps$, we first derive some asymptotic estimates on $\bigu(r)$.
	\begin{lemma}\label{lem-1}
		There exist positive constants $C_2$ and $M_2$ independent of $\eps$ such that, as $0<\eps\ll1$,
		\begin{align}
		|\mathtt{K}_{\eps}|\leq&\,{C}_2e^{-\frac{M_2}{\eps}R},\label{k-0}\\
		0<\left(\frac{r}{R}\right)^{N-1}\bigu'(r)\leq&\frac{C_2}{\eps}e^{-\frac{M_2}{\eps}(R-r)},\,\,r\in(0,R],\label{gra-est-u}
		\end{align}
		where $\mathtt{K}_{\eps}$ is defined in \eqref{0710-id3}.
		Moreover, there holds
		\begin{align}\label{crucial-estimate}
		\lim_{\eps\rightarrow0}\sup_{r\in[0,R]}\left|\sqrt{\eps}\bigu'(r)-\sqrt{\frac{2\wmr}{\eps}F(\bigu(r))}\right|<\infty.
		\end{align}
	\end{lemma}
	\begin{proof}
		Multiplying \eqref{eq3} by $r^{N-1}$, we obtain $\eps^2(r^{N-1}\bigu'(r))'={\meps}r^{N-1}f(\bigu)>0$. Hence, $r^{N-1}\bigu'(r)$ is strictly increasing with respect to $r$ due to the fact $0<\bigu(r)\leq u_0$. Since $\bigu'(0)=0$, we immediately obtain $\bigu'(r)>0$ for $r\in(0,R]$, which gives the left-hand side of \eqref{gra-est-u}.

		By \eqref{eq3} and \eqref{mu-asy}, one may check that, as $0<\eps\ll1$,
		\begin{align}\label{0710-id1}
		\eps^2(r^{N-1}\bigu')''={\meps}\left((N-1)r^{N-2}f(\bigu)+r^{N-1}f'(\bigu)\bigu'\right)\geq\widetilde{M}_2r^{N-1}\bigu',\,\,r\in[0,R],
		\end{align}
		where $\widetilde{M}_2$ is a positive constant close to $m$. Here we have used the fact that $\bigu'\geq0$, $f(\bigu)\geq0$ and
		$f'(\bigu)=(\bigu+1)e^{\bigu}\geq1$ to obtain the last inequality of \eqref{0710-id1}. Note also that $\bigu'(R)>\bigu'(0)=0$. Thus the standard comparison theorem applied to \eqref{0710-id1} shows
		\begin{align}\label{0710-id2}
		r^{N-1}\bigu'(r)\leq{R}^{N-1}\bigu'(R)e^{-\frac{\sqrt{\widetilde{M}_2}}{\eps}(R-r)},\,\,r\in[0,R].
		\end{align}
		Let us now estimate $\bigu'(R)$. By \eqref{0710-id3} and \eqref{0710-id2}, we have
		\begin{align}\label{0710-id5}
		0\leq\meps{F(\bigu(\frac{R}{2}))}+\mathtt{K}_{\eps}=\frac{\eps^2}{2}\bigu'^2(\frac{R}{2})\leq\eps^22^{2N-3}\bigu'^2(R)e^{-\frac{\sqrt{\widetilde{M}_2}}{\eps}R},
		\end{align}
		and for $r\in[\frac{R}{2},R]$,
		\begin{align}\label{0710-id5-0}
		\eps^2\int_{\frac{R}{2}}^r\frac{N-1}{s}\bigu'^2(s)\leq
		\frac{2^{2N-1}\eps^2}{R}(N-1)\bigu'^2(R)\int_{\frac{R}{2}}^Re^{-\frac{2\sqrt{\widetilde{M}_2}}{\eps}(R-s)}\leq
		C_3\eps^3\bigu'^2(R),
		\end{align}
		where $C_3$ is a positive constant independent of $\eps$. On the other hand, by \eqref{fts} and \eqref{u-est2}, we find
		\begin{align}\label{0710-id6}
		0\leq{F}(\bigu(\frac{R}{2}))\leq\left(\max_{r\in[0,R]}f(\bigu(r))\right)\bigu(\frac{R}{2})
		\leq C_4e^{-\frac{M_1}{2\eps}R},
		\end{align}
		where $C_4={u_0}e^{u_0}C_1$. Hence, by \eqref{mu-asy}, \eqref{0710-id5} and  \eqref{0710-id6} we obtain
		\begin{align}\label{0710-id7}
		|\mathtt{K}_{\eps}|\leq{C}_5\Big(e^{-\frac{M_1}{2\eps}R}+\eps^2\bigu'^2(R)e^{-\frac{\sqrt{\widetilde{M}_2}}{\eps}R}\Big),
		\end{align}
		where ${C}_5$ is a positive constant independent of $\eps$. As a consequence, by \eqref{0710-id3} and \eqref{0710-id7}, for sufficiently small $\eps>0$, we arrive at $\frac{\eps^2}{2}\bigu'^2(R)\leq\meps{F(\bigu(R))}+\mathtt{K}_{\eps}\leq2{\wmr}F(u_0)+C_5\eps^2\bigu'^2(R)e^{-\frac{\sqrt{\widetilde{M}_2}}{\eps}R}$.
		Since $e^{-\frac{\sqrt{\widetilde{M}_2}}{\eps}R}\ll1$, we get
		\begin{align}\label{0710-id9}
		0<\bigu'(R)\leq\frac{2}{\eps}\sqrt{{\wmr}F(u_0)}\,\,\mathrm{as}\,\,0<\eps\ll1.
		\end{align}
		Hence, \eqref{k-0} is obtained by \eqref{0710-id7} and \eqref{0710-id9}. The right-hand inequality of \eqref{gra-est-u} thus follows from \eqref{0710-id2} and \eqref{0710-id9}, where we set $C_2=\max\{C_5,2\sqrt{{\wmr}F(u_0)},4C_5{\wmr}F(u_0)\}$ and $M_2=\min\{\frac{M_1}{2},\sqrt{\widetilde{M}_2}\}$.

		It remains to prove \eqref{crucial-estimate}. Firstly, we give an estimate of $\meps-{\wmr}$ with respect to small $\eps>0$. Since $0<\bigu\leq{u_0}$, together with \eqref{u-est2} gives
		\begin{align*}
		\left|\frac{N}{R^N}\int_0^Re^{\bigu(s)}s^{N-1}\dss-1\right|=~&\left|\frac{N}{R^N}\int_0^R\left(e^{\bigu(s)}-1\right)s^{N-1}\dss\right|\notag\\
		\leq~&\frac{N(e^{u_0}-1)}{Ru_0}\int_0^R\bigu(s)\,\dss
		\leq\frac{N(e^{u_0}-1)C_1}{Ru_0M_1}\eps.
		\end{align*}
		Along with \eqref{coef} and \eqref{mu-asy}, one may check that
		\begin{align}\label{2day-2}
		|\meps-{\wmr}|=\wmr\left|\left(\frac{N}{R^N}\int_0^Re^{\bigu(s)}s^{N-1}\dss\right)^{-1}-1\right|\leq\frac{2{\wmr}N(e^{u_0}-1)C_1}{Ru_0M_1}\eps,\,\,\mathrm{as}\,\,0<\eps\ll1.
		\end{align}
		Combining \eqref{2day-2} with \eqref{0710-id3}, \eqref{0710-id5-0} and \eqref{0710-id9}, we arrive at, for $r\in[\frac{R}{2},R]$,
		\begin{align}\label{2day-3}
		\left|{\eps}\bigu'^2(r)-\frac{2{\wmr}F(\bigu(r))}{\eps}\right|\leq2\eps\int_{\frac{R}{2}}^r\frac{N-1}{s}\bigu'^2(s)\,\dss+\frac{2}{\eps}|\meps-{\wmr}|+\frac{2}{\eps}|\mathtt{K}_{\eps}|\leq{C}_6,
		\end{align}
		where $C_6$ is a positive constant independent of $\eps$. In particular, due to $\bigu'\geq0$ and $F(\bigu)\geq0$, \eqref{2day-3} implies
		\begin{align}\label{0716-f1}
		\Big|\sqrt{\eps}\bigu'(r)-\sqrt{\frac{2{\wmr}F(\bigu(r))}{\eps}}\Big|\leq\sqrt{C_6},\,\,\mathrm{for}\,\,r\in[\frac{R}{2},R].
		\end{align}
		On the other hand, by \eqref{u-est2} and \eqref{gra-est-u}, it is easy to see that
		\begin{align}\label{0716-f2}
		\sqrt{\eps}\bigu'(r)-\sqrt{\frac{2{\wmr}F(\bigu(r))}{\eps}}\stackrel{\eps\rightarrow0}{\longrightarrow}0\,\,\mathrm{uniformly\,\,in}\,\,[0,\frac{R}{2}].
		\end{align}
		Therefore, \eqref{crucial-estimate} follows from \eqref{0716-f1}--\eqref{0716-f2} and the proof of Lemma~\ref{lem-1} is complete.
	\end{proof}
	
	Setting $r=R$ in \eqref{crucial-estimate} and using $\bigu(R)=u_0$, we obtain
	\begin{align}\label{sec-es}
	\lim_{\eps\rightarrow0}{\eps}\bigu'(R)=\sqrt{{2{\wmr}}F(u_0)}
	\end{align}
	which gives the precise leading order term of $\bigu(R)$ as $0<\eps\ll1$. Note also from \eqref{2day-2} that ${\eps}^{-1}(\meps-{\wmr})$ is bounded for $0<\eps\ll1$. To further exploit ${\eps}^{-1}(\meps-{\wmr})$ so that we can get its precise leading order term, let us introduce the following approximation which essentially comes from the \textbf{Poho\v{z}aev-type identity} applied to \eqref{eq3}--\eqref{bd3}. Moreover, this result gives a relation between the second order term of $\meps$ and asymptotics of $\bigu$.
	
	\begin{lemma}\label{lem-2}
		As $0<\eps\ll1$, there holds
		\begin{align}\label{poho-mm}
		\frac{\meps-{\wmr}}{\eps}=-\frac{N}{R}\sqrt{2{\wmr}F(u_0)}+\eps\int_{\frac{R}{2}}^Rg(r)\bigu'^2(r)\,\dr+o_{\eps}(1),
		\end{align}
		where
		\begin{align}\label{poho-gg}
		g(r)=\frac{N-1}{r}-\frac{N-2}{2R^N}r^{N-1}.
		\end{align}
	\end{lemma}
	
	\begin{proof}
		Multiplying \eqref{0710-id3} by $r^{N-1}$ and integrating the expression over the interval $[0,R]$, we then have
		
		\begin{align}\label{0710-id3++}
		\underbrace{\frac{\eps^2}{2}\int_0^R\bigu'^2(r)r^{N-1}\dr+\eps^2\int_0^Rr^{N-1}\int_{\frac{R}{2}}^r\frac{N-1}{s}\bigu'^2(s)\,\dss\dr}_{:=\,\mathrm{P}_{\mathrm{I}}}=\meps\int_0^R{F(\bigu(r))}r^{N-1}\dr+\frac{R^N}{N}\mathtt{K}_{\eps}.
		\end{align}
		Using integration by parts,
		\begin{align*}
		\int_0^Rr^{N-1}\int_{\frac{R}{2}}^r\frac{N-1}{s}\bigu'^2(s)\,\dss\dr=\frac{(N-1)R^N}{N}\int_{\frac{R}{2}}^R\frac{1}{r}\bigu'^2(r)\,\dr-\frac{N-1}{N}\int_0^Rr^{N-1}\bigu'^2(r)\,\dr,
		\end{align*}
		one finds that
		\begin{align}\label{711-p1}
		&\left|\mathrm{P}_{\mathrm{I}}-\eps^2\int_{\frac{R}{2}}^R\frac{1}{r}\left(\frac{N-1}{N}R^N-\frac{N-2}{2N}r^{N}\right)\bigu'^2(r)\,\dr\right|\notag\\
		&\hspace{36pt}=\eps^2\int_0^{\frac{R}{2}}\frac{N-2}{2N}r^{N-1}\bigu'^2(r)\,\dr\leq\frac{N-2}{2N}R^{N-1}C_2\eps{e}^{-\frac{M_2R}{2\eps}}\int_0^{\frac{R}{2}}\bigu'(r)\,\dr\\
		&\hspace{36pt}\leq{C}_7\eps{e}^{-\frac{M_2R}{2\eps}}.\notag
		\end{align}
		Here we have used \eqref{gra-est-u} to obtain $r^{N-1}\bigu'^2(r)\leq\frac{C_2}{\eps}R^{N-1}{e}^{-\frac{M_2R}{2\eps}}\bigu'(r)$
		for $r\in[0,\frac{R}{2}]$, which is used to deal with the inequality in the second line of \eqref{711-p1}.

		Next, we deal with the right-hand side of \eqref{0710-id3++}. By \eqref{fts}--\eqref{coef} and \eqref{sec-es}, one obtains
		\begin{align}\label{big-f}
		\meps\int_0^R{F(\bigu(r))}r^{N-1}\dr=&\,\meps\int_0^R\left(1-e^{\bigu}+f(\bigu(r))\right)r^{N-1}\dr\notag\\
		=&\,\frac{R^N}{N}\left(\meps-{\wmr}\right)+\eps^2R^{N-1}\bigu'(R)\\
		=&\,\frac{R^N}{N}\left(\meps-{\wmr}\right)+\eps\left({R}^{N-1}\sqrt{2{\wmr}F(u_0)}+o_{\eps}(1)\right).\notag
		\end{align}
		Here we have used identities ${\meps}f(\bigu(r))r^{N-1}=\eps^2\left(r^{N-1}\bigu'\right)'$ and $\meps\int_0^Re^{\bigu}r^{N-1}\mathrm{d}r=\frac{m}{N\alpha(N)}=\frac{R^N}{N}{\wmr}$ to get the second line of \eqref{big-f}. As a consequence,
		by \eqref{0710-id3++} and \eqref{big-f}, we have
		\begin{align}\label{0716-af}
		\mathrm{P}_{\mathrm{I}}=\eps\left({R}^{N-1}\sqrt{2{\wmr}F(u_0)}+o_{\eps}(1)\right)+\frac{R^N}{N}\left(\meps-{\wmr}+\mathtt{K}_{\eps}\right).
		\end{align}
		By \eqref{0710-id7}, \eqref{711-p1} and \eqref{0716-af},
		after making appropriate manipulations it yields
		\begin{equation}
		\begin{aligned}
		\label{p-meps=n}
		&\left|\frac{\meps-{\wmr}}{\eps}+\frac{N}{R}\sqrt{2{\wmr}F(u_0)}
		-\eps\int_{\frac{R}{2}}^R\left(\frac{N-1}{r}-\frac{N-2}{2R^N}r^{N-1}\right)\bigu'^2(r)\,\dr\right|\\
		&\leq\frac{|\mathtt{K}_{\eps}|}{\eps}+\frac{C_7N}{R^N}e^{-\frac{M_2R}{2\eps}}+o_{\eps}(1)\to0,
		\end{aligned}
		\end{equation}
		as $\eps\to0$.
		Therefore, \eqref{p-meps=n} implies \eqref{poho-mm} and the proof of Lemma~\ref{lem-2} is completed.
	\end{proof}
	
	We are now in a position to establish the precise leading order term of $\meps-{\wmr}$ for small $\eps>0$.
	\begin{proposition}[Refined estimate of ${\meps}$]\label{pro-m}
		As $0<\eps\ll1$,  the asymptotic expansion of ${\meps}$ with precise first two order terms involving the effect of curvature $R^{-1}$ is described as follows:
		\begin{align}\label{cmh}
		{\meps}={\wmr}+\eps\frac{N}{R}\sqrt{\wmr}\left(\hnu+o_{\eps}(1)\right)\,\,as\,\,0<\eps\ll1,
		\end{align}
		where $\hnu=-\sqrt{2F(u_0)}+\int_0^{u_0}\sqrt{\frac{F(t)}{2}}\,\mathrm{d}t$ defined in \eqref{c-0} depends mainly on the boundary value $u_0$ and is independent of $R$. Moreover, $\hnu<0$ is a strictly decreasing function of $u_0\in(0,\infty)$.
	\end{proposition}
	\begin{proof}
		By Lemma~\ref{lem-2}, it suffices to obtain the precise leading order term of
		\begin{align}\label{716-4eps}
		\mathrm{P}_{\mathrm{I}\!\mathrm{I}}:=\eps\int_{\frac{R}{2}}^R\left(\frac{N-1}{r}-\frac{N-2}{2R^N}r^{N-1}\right)\bigu'^2(r)\,\dr.
		\end{align}
		Thanks to \eqref{gra-est-u}, we shall consider the decomposition of \eqref{716-4eps} as
		\begin{align}\label{0711-1002}
		\mathrm{P}_{\mathrm{I}\!\mathrm{I}}=\eps\left\{\int_{\frac{R}{2}}^{R-\sqrt{\eps}}+\int_{R-\sqrt{\eps}}^R\right\}\left(\frac{N-1}{r}-\frac{N-2}{2R^N}r^{N-1}\right)\bigu'^2(r)\,\dr.
		\end{align}
		In particular, we have
		\begin{align}\label{0711-1024}
		\eps&\,\left|\int_{\frac{R}{2}}^{R-\sqrt{\eps}}\left(\frac{N-1}{r}-\frac{N-2}{2R^N}r^{N-1}\right)\bigu'^2(r)\,\dr\right|\notag\\[-0.7em]&\\[-0.7em]
		&\hspace{12pt}\leq\frac{2\eps(N-1)}{R}\int_{\frac{R}{2}}^{R-\sqrt{\eps}}\bigu'^2(r)\,\dr\leq\frac{2^{2(N-1)}(N-1)C_2^2}{M_2R}e^{-\frac{2M_2}{\sqrt{\eps}}}.\notag
		\end{align}
		To deal with the second integral of $\mathrm{P}_{\mathrm{I}\!\mathrm{I}}$, let us set
		\begin{align*}
		\xi_{\eps}(r)=\left(\frac{N-1}{r}-\frac{N-2}{2R^N}r^{N-1}\right)-\frac{N}{2R},\,\,r\in[R-\sqrt{\eps},R].
		\end{align*}
		It is easy to get
		$\ds\sup_{r\in[R-\sqrt{\eps},R]}|\xi_{\eps}(r)|\leq{C}_8\sqrt{\eps}$.
		This along with \eqref{gra-est-u} immediately gives
		\begin{align}\label{0711-1108}
		\eps\left|\int_{R-\sqrt{\eps}}^R\xi_{\eps}(r)\bigu'^2(r)\,\dr\right|\leq{C}_9\sqrt{\eps}.
		\end{align}
		Here $C_8$ and $C_9$ are positive constants independent of $\eps$.

		On the other hand, by \eqref{crucial-estimate} we have
		\begin{align}\label{0711-1006}
		{\eps}\bigu'(r)=\sqrt{2{\wmr}F(\bigu(r))}+\sqrt{\eps}\gamma_{\eps}(r)\,\,\mathrm{with}\,\,\lim_{\eps\rightarrow0}\sup_{[0,R]}\left|\gamma_{\eps}(r)\right|<\infty.
		\end{align}
		Using \eqref{0711-1108} and \eqref{0711-1006}, one may check that
		\begin{align}\label{0711-1130}
		\eps\int_{R-\sqrt{\eps}}^R&\left(\frac{N-1}{r}-\frac{N-2}{2R^N}r^{N-1}\right)\bigu'^2(r)\,\dr\notag\\
		&\quad=\eps\frac{N}{2R}\int_{R-\sqrt{\eps}}^R\bigu'^2(r)\,\dr+\eps\int_{R-\sqrt{\eps}}^R\xi_{\eps}(r)\bigu'^2(r)\,\dr\notag\\
		&\quad=\frac{N}{2R}\int_{R-\sqrt{\eps}}^R\left(\sqrt{2{\wmr}F(\bigu(r))}+\sqrt{\eps}\gamma_{\eps}(r)\right)\bigu'(r)\,\dr+\eps\int_{R-\sqrt{\eps}}^R\xi_{\eps}(r)\bigu'^2(r)\,\dr\\
		&\quad=\frac{N}{2R}\int_{\bigu(R-\sqrt{\eps})}^{u_0}\sqrt{2{\wmr}F(t)}\,\mathrm{d}t+o_{\eps}(1)\notag\\
		&\quad=\frac{N}{2R}\int_{0}^{u_0}\sqrt{2{\wmr}F(t)}\,\mathrm{d}t+o_{\eps}(1).\notag
		\end{align}
		Here we stress that in the last two lines of \eqref{0711-1130}, we have verified
		\begin{align*}
		\left|\int_{R-\sqrt{\eps}}^R\sqrt{\eps}\gamma_{\eps}(r)\bigu'(r)\,\dr\right|\leq\sqrt{\eps}\sup_{[R-\sqrt{\eps},R]}|\gamma_{\eps}(r)|(\bigu(R)-\bigu(R-\sqrt{\eps}))\to0
		\end{align*}
		and
		\begin{align*}
		\int_0^{\bigu(R-\sqrt{\eps})}\sqrt{2{\wmr}F(t)}\,\mathrm{d}t\leq\sqrt{2{\wmr}F(u_0)}\bigu(R-\sqrt{\eps})\to0
		\end{align*}
		as $\eps\rightarrow0$ (by \eqref{u-est2}).
		As a consequence, by \eqref{0711-1002}, \eqref{0711-1024} and \eqref{0711-1130}, we obtain the precise leading order term of $\mathrm{P}_{\mathrm{I}\!\mathrm{I}}$,
		\begin{align}\label{0711-1130-ADD}
		\mathrm{P}_{\mathrm{I}\!\mathrm{I}}=\frac{N}{2R}\int_{0}^{u_0}\sqrt{2{\wmr}F(t)}\,\mathrm{d}t+o_{\eps}(1).
		\end{align}
		Finally, by \eqref{poho-mm}--\eqref{poho-gg} and \eqref{0711-1130-ADD}, we get
		\begin{align*}
		\frac{\meps-{\wmr}}{\eps}=-\frac{N}{R}\sqrt{2{\wmr}F(u_0)}+\mathrm{P}_{\mathrm{I}\!\mathrm{I}}+o_{\eps}(1)=\frac{N}{R}\sqrt{{\wmr}}\bigg(\underbrace{\int_{0}^{u_0}\sqrt{\frac{F(t)}{2}}\,\mathrm{d}t-\sqrt{2F(u_0)}}_{:=\hnu}+o_{\eps}(1)\bigg).
		\end{align*}
		This along with \eqref{mu-asy} gives \eqref{cmh}.
		
		It remains to prove
		\begin{align}\label{u_0-add}
		\hnu<0\,\,\mathrm{and}\,\,\frac{\mathrm{d}\pmb{\mathrm{J}}}{\mathrm{d}{u_0}}(u_0)<0\,\,\mathrm{for}\,\,u_0>0.
		\end{align}
		Indeed, by a simple calculation we get $\pmb{\mathrm{J}}(0)=0$ and
		\begin{align*}
		\frac{\mathrm{d}\pmb{\mathrm{J}}}{\mathrm{d}{u_0}}(u_0)=\frac{F(u_0)-f(u_0)}{\sqrt{2F(u_0)}}=\frac{1-e^{u_0}}{\sqrt{2F(u_0)}}<0,
		\end{align*}
		{which implies} \eqref{u_0-add}.
		Therefore, we complete the proof of
		Proposition~\ref{pro-m}.
	\end{proof}
	
	\begin{remark}
		Proposition~\ref{pro-m} also shows the effect of boundary value $u_0$ on $\meps$. Precisely speaking, let $R>0$ be fixed and $u_0\in[l_1,l_2]$, where $0<l_1<l_2<\infty$. Regarding $\meps$ as a function of $u_0$, we find that as $0<\eps\ll\frac{R}{|\pmb{\mathrm{J}}(l_1)|}\sqrt{\wmr}$,
		$\meps$ is strictly decreasing to $u_0\in[l_1,l_2]$, where $\pmb{\mathrm{J}}(l_1):=-\sqrt{2F(l_1)}+\int_0^{l_1}\sqrt{\frac{F(t)}{2}}\,\mathrm{d}t$.
	\end{remark}

	\subsection{Proof of Theorem~\ref{mainthm}}\label{proof-thm-u}
	We first establish the following result.
	\begin{lemma}\label{thm-u-lem-1}
		Let $\hnu$ be as defined in \eqref{c-0}. Then For each $j>0$ independent of $\eps$, we have
		\begin{align}\label{newu-0712}
		\lim_{\eps\rightarrow0}\sup_{r_{\eps}\in[R-j\eps,R]}&\left|\bigu'(r_{\eps})-\left\{\frac{\sqrt{2{\wmr}F(\bigu(r_{\eps}))}}{\eps}\right.\right.\notag\\[-0.8em]
		&\\[-0.8em]
		&\hspace{10pt}\left.\left.+\frac{1}{R}\left(N{\hnu}\sqrt{\frac{F(\bigu(r_{\eps}))}{2}}-(N-1)\int_0^{\bigu(r_{\eps})}\sqrt{\frac{F(t)}{F(\bigu(r_{\eps}))}}\,\mathrm{d}t\right)\right\}\right|=0.\notag
		\end{align}
	\end{lemma}
	\begin{proof}
		By Corollary \ref{th1.3}-(ii), we have
		\begin{align}\label{fi-0718}
		\lim_{\eps\rightarrow0}\inf_{r_{\eps}\in[R-j\eps,R]}\bigu(r_{\eps})>0.
		\end{align}
		Setting $r=r_{\eps}$ in \eqref{0710-id3}, using \eqref{cmh} and following the similar argument as in \eqref{0711-1002}--\eqref{0711-1130}, one may check that
		\begin{align}\label{0718-339}
		{\eps^2}\bigu'^2(r_{\eps})=&\,2\left({\meps}{F(\bigu(r_{\eps}))}-\eps^2\int_{\frac{R}{2}}^{r_{\eps}}\frac{N-1}{s}\bigu'^2(s)\,\dss+\mathtt{K}_{\eps}\right)\notag\\
		=&\,2\left(m+\eps\sqrt{{\wmr}}\left(\frac{N{\hnu}}{R}+o_{\eps}(1)\right)\right){F(\bigu(r_{\eps}))}\notag\\
		&\quad-2\eps\left(\frac{N-1}{R}+o_{\e}(1)\right)\left(\int_{\bigu(\frac{R}{2})}^{\bigu(r_{\eps})}\sqrt{2{\wmr}F(t)}\mathrm{d}t+o_{\eps}(1)\right)+2\mathtt{K}_{\eps}\\
		=&\,2{\wmr}F(\bigu(r_{\eps}))+\frac{2\eps\sqrt{{\wmr}}}{R}\left(N{\hnu}{F(\bigu(r_{\eps}))}-(N-1)\int_0^{\bigu(r_{\eps})}\sqrt{2F(t)}\,\mathrm{d}t+o_{\eps}(1)\right)\notag\\
		=&\,2{\wmr}F(\bigu(r_{\eps}))\left\{1+\frac{\eps}{R\sqrt{{\wmr}}}\left(N{\hnu}-(N-1)\int_0^{\bigu(r_{\eps})}\frac{\sqrt{2F(t)}}{F(\bigu(r_{\eps}))}\,\mathrm{d}t+o_{\eps}(1)\right)\right\}.\notag
		\end{align}
		Due to \eqref{0711-1006} and \eqref{fi-0718}, the asymptotic expansions in \eqref{0718-339} is uniformly in $[R-j\eps,R]$ as $0<\eps\ll1$. Since $\bigu'\geq0$, by \eqref{fi-0718} and \eqref{0718-339} we have
		\begin{align}\label{cha-0718}
		\bigu'(r_{\eps})=\frac{\sqrt{2{\wmr}F(\bigu(r_{\eps}))}}{\eps}\left\{1+\frac{\eps}{2R\sqrt{{\wmr}}}\left(N{\hnu}-(N-1)\int_0^{\bigu(r_{\eps})}\frac{\sqrt{2F(t)}}{F(\bigu(r_{\eps}))}\,\mathrm{d}t+o_{\eps}(1)\right)\right\}
		\end{align}
		uniformly in $[R-j\eps,R]$ as $0<\eps\ll1$. This gives \eqref{newu-0712} and completes the proof of Lemma~\ref{thm-u-lem-1}.
	\end{proof}
	
	By \eqref{0716-f1}, \eqref{fi-0718} and \eqref{cha-0718}, we have
	\begin{align}\label{0718-528}
	\left|\frac{\bigu'(r)}{\sqrt{2{\wmr}F(\bigu(r))}}-\frac{1}{\eps}\right|\leq{C_{10}(j,{R})},\,\,\mathrm{for}\,\,r\in[R-j\eps,R],
	\end{align}
	where $C_{10}(j,{R})$ (depending mainly on $j$ and ${R}$) is a positive constant independent of $\eps$. In particular, for $j>d_0$, let us integrate \eqref{0718-528} over $[r_{\eps}(d_0),R]$ with $r_{\eps}(d_0)=R-d_0\eps$, which results in
	\begin{align}\label{0718-545}
	\left|\int_{\bigu(r_{\eps}(d_0))}^{u_0}\frac{\mathrm{d}t}{\sqrt{2{\wmr}F(t)}}-d_0\right|\leq{C}_{10}(j,R)d_0{\eps}.
	\end{align}
	\textcolor[rgb]{0,0,0}{Moreover, let $\Phi$ denote the unique positive solution of the equation
		\begin{align}\label{phi-eq819}
		\begin{cases}
		-\Phi'(t)=\sqrt{2\rho_0F(t)},\,\,t>0,\\
		\Phi(0)=u_0,\,\,\Phi(\infty)=0.
		\end{cases}
		\end{align}
		Then for $d_0>0$, \eqref{phi-eq819} directly implies}
	\begin{align}\label{0718-d0}
	d_0=\int_{\Phi(d_0)}^{u_0}\frac{\mathrm{d}t}{\sqrt{2{\wmr}F(\Phi(t))}}.
	\end{align}
	This along with \eqref{0718-545} immediately yields $\int_{\bigu(r_{\eps}(d_0))}^{\Phi(d_0)}\frac{\mathrm{d}t}{\sqrt{2{\wmr}F(t)}}\stackrel{\eps\rightarrow0}{\longrightarrow}0$. Moreover,
	$\bigu(r_{\eps}(d_0))\stackrel{\eps\rightarrow0}{\longrightarrow}\Phi(d_0)$
	since $\frac{1}{\sqrt{2{\wmr}F(t)}}$ has a positive lower bound in $t\in[\Phi(d_0),u_0]$. As a consequence,
	\begin{align}\label{0718-6pm}
	\bigu(r_{\eps}(d_0))=\Phi(d_0)+L_{\eps}(d_0),\,\,\lim_{\eps\rightarrow0}L_{\eps}(d_0)=0.
	\end{align}

	\textcolor[rgb]{0,0,0}{On the other hand, by \eqref{phi-eq}, \eqref{mu-asy} and \eqref{phi-eq819} and the uniqueness of $\Psi$ and $\Phi$, we have
		$\Phi(t)=\Psi(\sqrt{\frac{\rho_0}{m}}t)$ with $\frac{\rho_0}{m}=\frac{1}{\alpha(N)R^N}$. Since $\Phi$ depends on $R$,  for the convenience of our next arguments, we shall denote
		\begin{align}\label{tmxi}
		\Phi(t):=\Psi^R(t)=\Psi(\frac{t}{\sqrt{\alpha(N)}R^{N/2}}).
		\end{align}}

	Then we are able to claim the following result.
	\begin{lemma}\label{lem-0718-6pm}
		As $0<\eps\ll1$,
		\begin{align}\label{leps-0718}
		\frac{L_{\eps}(d_0)}{\eps}=-\frac{\sqrt{2F(\Psi^R(d_0))}}{2R}\left(d_0N\hnu-\frac{N-1}{\sqrt{\rho_0}}{\jnu}+o_{\eps}(1)\right).
		\end{align}
	\end{lemma}
	\begin{proof}
		We shall follow the similar argument as in the proof of \cite[Theorem~4.1(III)]{L2019} and \cite[Lemma~4.1]{L2019-s}. Let $j>d_0$ in \eqref{fi-0718}. By \eqref{cha-0718} we have,  as $0<\eps\ll1$,
		\begin{align}\label{cha-0718++}
		\frac{\bigu'(r_{\eps})}{\sqrt{2{\wmr}F(\bigu(r_{\eps}))}}=\frac{1}{\eps}+\frac{1}{2R\sqrt{{\wmr}}}\left(N{\hnu}-(N-1)\int_0^{\bigu(r_{\eps})}\frac{\sqrt{2F(t)}}{F(\bigu(r_{\eps}))}\,\mathrm{d}t\right)+o_{\eps}(1)
		\end{align}
		uniformly in $[R-j\eps,R]$. Therefore, by integrating \eqref{cha-0718++} over $[r_{\eps}(d_0),R]\,(\subset[R-j\eps,R])$, one arrives at
		\begin{align}\label{cha-0718c7}
		&\int_{\bigu(r_{\eps}(d_0))}^{u_0}\frac{\mathrm{d}t}{\sqrt{2{\wmr}F(t)}}\notag\\[-0.7em]
		&\\[-0.7em]
		=&\,d_0+\frac{1}{2R\sqrt{{\wmr}}}\left(d_0N{\hnu}\eps-(N-1)\int_{R-d_0\eps}^R\int_0^{\bigu(s)}\frac{\sqrt{2F(t)}}{F(\bigu(s))}\,\mathrm{d}t\mathrm{d}s\right)+{\eps}o_{\eps}(1).\notag
		\end{align}
		With a simple calculation, we obtain
		\begin{align}\label{0801-a}
		\int_{\bigu(r_{\eps}(d_0))}^{u_0}\frac{\mathrm{d}t}{\sqrt{2{\wmr}F(t)}}=&\,\left\{\int_{\Psi^R(d_0)}^{u_0}+\int_{\Psi^R(d_0)+L_{\eps}(d_0)}^{\Psi^R(d_0)}\right\}\frac{\mathrm{d}t}{\sqrt{2{\wmr}F(t)}}\notag\\[-0.7em]
		&\\[-0.7em]
		=&\,d_0-\frac{L_{\eps}(d_0)}{\sqrt{2{\wmr}F(\Psi^R(d_0))}}(1+o_{\eps}(1)).\notag
		\end{align}
		Here we have used \eqref{0718-d0}--\eqref{tmxi} to get the first and the second terms in the last line.

		On the other hand, by using \eqref{0718-528} with $j>d_0$, we can deal with the last integral of the right-hand side of \eqref{cha-0718c7} as follows:
		\begin{align}\label{0718-nigh}
		&\int_{R-d_0\eps}^R\int_0^{\bigu(s)}\frac{\sqrt{2F(t)}}{F(\bigu(s))}\,\mathrm{d}t\,\mathrm{d}s\notag\\
		=&\int_{R-d_0\eps}^R\left(\frac{\eps\bigu'(s)}{\sqrt{2{\wmr}F(\bigu(s))}}+o_{\eps}(1)\right)\int_0^{\bigu(s)}\frac{\sqrt{2F(t)}}{F(\bigu(s))}\mathrm{d}t\,\mathrm{d}s\\
		=&\int_{\bigu(R-d_0\eps)}^{u_0}\frac{\eps}{\sqrt{2{\wmr}F(\widetilde{s})}}\int_0^{\widetilde{s}}\frac{\sqrt{2F(t)}}{F(\widetilde{s})}\mathrm{d}t\,{\mathrm{d}\widetilde{s}}+{\eps}o_{\eps}(1)\notag\\
		=&\int_{\Psi^R(d_0)}^{u_0}\frac{\eps}{\sqrt{2{\wmr}F(\widetilde{s})}}\int_0^{\widetilde{s}}\frac{\sqrt{2F(t)}}{F(\widetilde{s})}\mathrm{d}t\,{\mathrm{d}\widetilde{s}}+{\eps}o_{\eps}(1).\notag
		\end{align}
		Here we have used \eqref{fi-0718} and \eqref{0718-6pm} to verify that $\int_0^{\bigu(s)}\frac{\sqrt{2F(t)}}{F(\bigu(s))}\mathrm{d}t\leq\frac{\sqrt{2}\bigu(s)}{\sqrt{F(\bigu(s))}}$ is uniformly bounded for $s\in[R-d_0\eps,R]$, and
		$$\int_{\Psi^R(d_0)}^{\Psi^R(d_0)+L_{\eps}(d_0)}\frac{\eps}{\sqrt{2{\wmr}F(\widetilde{s})}}\int_0^{\widetilde{s}}\frac{\sqrt{2F(t)}}{F(\widetilde{s})}\mathrm{d}t\,{\mathrm{d}\widetilde{s}}=\eps{o}_{\eps}(1).$$
		Combining \eqref{cha-0718c7}--\eqref{0801-a} with \eqref{0718-nigh} yields
		\begin{align*}
		\frac{L_{\eps}(d_0)}{\sqrt{2{\wmr}F(\Psi^R(d_0))}}=-\frac{\eps}{2R\sqrt{{\wmr}}}\left(d_0N{\hnu}-\frac{N-1}{\sqrt{{\wmr}}}\int_{\Psi^R(d_0)}^{u_0}\frac{1}{F(\widetilde{s})}\int_0^{\widetilde{s}}\sqrt{\frac{{F(t)}}{{F(\widetilde{s})}}}\mathrm{d}t{\mathrm{d}\widetilde{s}}+o_{\eps}(1)\right).
		\end{align*}
		This together with \eqref{wh-0717} implies \eqref{leps-0718}. Therefore, the proof of Lemma~\ref{lem-0718-6pm} is completed.
	\end{proof}
	
	Now we present an important result.
	
	\begin{proposition}[Asymptotics of $\bigu$ near boundary]\label{newt}
		Let $m$ and $u_0$ be positive constants independent of $\eps$, and let $r_{\eps}:=r_{\eps}(d_0)=R-d_0\eps\in(0,R]$ be a point with the distance $d_0\eps$ to the boundary, where $d_0\geq0$ is independent of $\eps$. Then \eqref{asys-u} holds, and we have
		\begin{align}\label{asys-u-diff}
		\bigu'(r_{\eps}(d_0))
		=&\sqrt{\frac{m}{\omega_{N}}}\left(\frac{\sqrt{2F(\Psi^R(d_0))}}{\eps}-\frac{d_0N}{2R}f(\Psi^R(d_0))\hnu\right)\notag\\[-0.7em]
		&\\[-0.7em]
		&+\frac{1}{R}\left(N\sqrt{\frac{F(\Psi^R(d_0))}{2}}{\hnu}+(N-1){\wjnu}\right)+o_{\eps}(1),\notag
		\end{align}
		where $\hnu$ and $\jnu$ are defined in \eqref{c-0} and \eqref{wh-0717}, respectively, and
		\begin{align}
		\wjnu=&\,\frac{1}{2}f(\Psi^R(d_0))\jnu-\int_0^{\Psi^R(d_0)}\sqrt{\frac{F(t)}{F(\Psi^R(d_0))}}\,\mathrm{d}t.\label{wh-0728}
		\end{align}
	\end{proposition}
	\begin{proof}
		The combination of \eqref{0718-6pm} and \eqref{leps-0718} yields \eqref{asys-u}. Next we want to prove $\eqref{asys-u-diff}$. Firstly, by \eqref{newu-0712} and \eqref{0718-6pm} we get
		\begin{align}\label{newu-0718}
		\bigu'(r_{\eps}(d_0))=~&\frac{1}{\eps}\sqrt{2{\wmr}F(\Psi^R(d_0)+L_{\eps}(d_0))}\notag\\[-0.7em]
		&\\[-0.7em]
		&\,+\frac{1}{R}\left(\sqrt{\frac{F(\Psi^R(d_0))}{2}}N{\hnu}-(N-1)\int_0^{\Psi^R(d_0)}\sqrt{\frac{F(t)}{F(\Psi^R(d_0))}}\,\mathrm{d}t+o_{\eps}(1)\right).\notag
		\end{align}
		Here we have used the approximation
		\begin{align}\label{0719-firs}
		F(\Psi^R(d_0)+L_{\eps}(d_0))=F(\Psi^R(d_0))+f(\Psi^R(d_0))L_{\eps}(d_0)(1+o_{\eps}(1))=F(\Psi^R(d_0))+o_{\eps}(1)
		\end{align}
		(by \eqref{0718-6pm}) to obtain the second line of \eqref{newu-0718}.
		
		Furthermore, to establish a refined asymptotics of $\bigu'(r_{\eps}(d_0))$ from \eqref{newu-0718}, obtaining the precise first two order terms of ${\eps}^{-1}{\sqrt{2{\wmr}F(\Psi^R(d_0)+L_{\eps}(d_0))}}$ is required since its second order term may be combined with the last term of \eqref{newu-0718}. By \eqref{leps-0718} and \eqref{0719-firs},
		one may use the approximation $\sqrt{1+\eta}\sim1+\frac{\eta}{2}$ (as $|\eta|\ll1$) to deal with this term as follows:
		\begin{align}\label{0719-135am}
		&\frac{1}{\eps}\sqrt{2{\wmr}F(\Psi^R(d_0)+L_{\eps}(d_0))}\notag\\
		&\quad\quad=\frac{1}{\eps}\sqrt{2{\wmr}[F(\Psi^R(d_0))+f(\Psi^R(d_0))L_{\eps}(d_0)(1+o_{\eps}(1))]}\notag\\[-0.7em]
		&\\[-0.7em]
		&\quad\quad=\frac{\sqrt{2{\wmr}F(\Psi^R(d_0))}}{\eps}\left(1+\frac{f(\Psi^R(d_0))}{2F(\Psi^R(d_0))}L_{\eps}(d_0)(1+o_{\eps}(1))\right)\notag\\
		&\quad\quad=\frac{\sqrt{2{\wmr}F(\Psi^R(d_0))}}{\eps}-\frac{f(\Psi^R(d_0))}{2R}\left(\sqrt{\wmr}d_0N\hnu-(N-1){\jnu}+o_{\eps}(1)\right),\notag
		\end{align}
		where $\jnu$ is defined in \eqref{wh-0717}.
		Consequently, by  \eqref{newu-0718} and \eqref{0719-135am}, one may check that
		{\small
			\begin{align*}
			\bigu'(r_{\eps}(d_0))
			=&~\frac{\sqrt{2{\wmr}F(\Psi^R(d_0))}}{\eps}-\frac{f(\Psi^R(d_0))}{2R}\left(\sqrt{\wmr}d_0N\hnu-(N-1){\jnu}\right)\\
			&\,+\frac{1}{R}\left(\sqrt{\frac{F(\Psi^R(d_0))}{2}}N{\hnu}-(N-1)\int_0^{\Psi^R(d_0)}\sqrt{\frac{F(t)}{F(\Psi^R(d_0))}}\,\mathrm{d}t\right)+o_{\eps}(1)\\
			=&~\sqrt{\wmr}\left(\frac{\sqrt{2F(\Psi^R(d_0))}}{\eps}-\frac{f(\Psi^R(d_0))}{2R}d_0N\hnu\right)
			+\frac{1}{R}\Bigg\{\sqrt{\frac{F(\Psi^R(d_0))}{2}}N{\hnu}\\
			&~+(N-1)\Bigg(\underbrace{\frac{f(\Psi^R(d_0))}{2}{\jnu}-\int_0^{\Psi^R(d_0)}\!\!\sqrt{\frac{F(t)}{F(\Psi^R(d_0))}}\,\mathrm{d}t}_{:=\wjnu\,\,(\mathrm{defined\,\,in}\,\,\eqref{wh-0728})}\Bigg)\Bigg\}+o_{\eps}(1).
			\end{align*}
		}
		This along with \eqref{mu-asy} gives \eqref{asys-u-diff}.
		Thus the proof of Proposition~\ref{newt} is complete.
	\end{proof}
	
	Since $c\in(0,u_0)$ is independent of $\eps$, by \eqref{gamma-thin}, \eqref{phi-eq819}--\eqref{0718-d0} and \eqref{tmxi} we know that
	\begin{align*}
	\frac{R-r_{\eps}(R,c)}{\eps}=~&(\Psi^R)^{-1}(c)+d_{1,\eps}(c)\notag\\[-0.7em]
	&\\[-0.7em]
	=~&\sqrt{\alpha(N)}R^{N/2}\Psi^{-1}(c)+d_{1,\eps}(c)\,\,\mathrm{with}\,\,\lim_{\eps\rightarrow0}d_{1,\eps}(c)=0.\notag
	\end{align*}
	Here we have used \eqref{tmxi} to verify $(\Psi^R)^{-1}(c)=\sqrt{\alpha(N)}R^{N/2}\Psi^{-1}(c)$.
	Furthermore, following the same argument as in Lemma~\ref{lem-0718-6pm}, we can obtain the asymptotics of $d_{1,\eps}(c)$ as follows:
	
	\begin{lemma}\label{lem-add8}
		For $R_0>0$, we have
		{\small
			\begin{align}\label{1055-0728}
			\lim_{\eps\rightarrow0}\sup_{R\in\textcolor[rgb]{0,0,0}{(0,R_0]}}\left|\frac{d_{1,\eps}(c)}{\eps}-\frac{\alpha(N)R^{N-1}}{2}\left(\textcolor[rgb]{0,0,0}{-\frac{N}{\sqrt{m}}\Psi^{-1}(c)\hnu}+\frac{N-1}{m}\int_{c}^{u_0}\left(\frac{1}{F(s)}\int_0^{s}\sqrt{\frac{{F(t)}}{{F(s)}}}\,\mathrm{d}t\right)\mathrm{d}s\right)\right|=0.
			\end{align}
		}
	\end{lemma}
	\begin{proof}
		\textcolor[rgb]{0,0,0}{For the simplicity of notations, in this proof we shall denote the inverse function of $\Psi^R$ (see \eqref{tmxi}) by $\Phi^{-1}$.}
		
		Firstly, we let $R>0$ be fixed. As $0<\eps\ll1$, we can set $j=2\Phi^{-1}(c)$ in \eqref{0718-528} and integrate \eqref{0718-528} over the interval $[R-\eps\left(\Phi^{-1}(c)+d_{1,\eps}(c)\right),R-\eps\Phi^{-1}(c)]$. As a consequence,
		\begin{align}\label{0803-343}
		d_{1,\eps}(c)=&\,(1+o_{\eps}(1))\int_{R-\eps\left(\Phi^{-1}(c)+d_{1,\eps}(c)\right)}^{R-\eps\Phi^{-1}(c)}\frac{\bigu'(r)}{\sqrt{2{\wmr}F(\bigu(r))}}\notag\\
		=&(1+o_{\eps}(1))\int_{\bigu(R-\eps\left(\Phi^{-1}(c)+d_{1,\eps}(c)\right))}^{\bigu(R-\eps\Phi^{-1}(c))}\frac{\mathrm{d}t}{\sqrt{2{\wmr}F(t)}}\notag\\
		=&\,-\frac{\eps}{R}\sqrt{\frac{F(c)}{2}}\left(\Phi^{-1}(c)N\hnu-\frac{N-1}{\sqrt{\rho_0}}\int_{c}^{u_0}\left(\frac{1}{F(s)}\int_0^{s}\sqrt{\frac{{F(t)}}{{F(s)}}}\,\mathrm{d}t\right)\mathrm{d}s+o_{\eps}(1)\right)\\
		&\times\left(\frac{1}{\sqrt{2{\wmr}F(c)}}+o_{\eps}(1)\right).\notag\\
		=&\frac{\eps}{2R}\left(-\frac{\Phi^{-1}(c)N}{\sqrt{\rho_0}}\hnu+{\frac{N-1}{\rho_0}}\int_{c}^{u_0}\left(\frac{1}{F(s)}\int_0^{s}\sqrt{\frac{{F(t)}}{{F(s)}}}\,\mathrm{d}t\right)\mathrm{d}s+o_{\eps}(1)\right)\notag\\
		=&\textcolor[rgb]{0,0,0}{\frac{\alpha(N)R^{N-1}{\e}}{2}\left({-\frac{N}{\sqrt{m}}\Psi^{-1}(c)\hnu}+\frac{N-1}{m}\int_{c}^{u_0}\left(\frac{1}{F(s)}\int_0^{s}\sqrt{\frac{{F(t)}}{{F(s)}}}\,\mathrm{d}t\right)\mathrm{d}s+o_{\eps}(1)\right)}.\notag
		\end{align}
		Here we have used $\bigu(R-\eps\left(\Phi^{-1}(c)+d_{1,\eps}(c)\right))=c$ and, by \eqref{asys-u},
		{\small
			\begin{align*}
			\bigu(R-\eps\Phi^{-1}(c))=c-\frac{\eps}{R}\sqrt{\frac{F(c)}{2}}\left(\Phi^{-1}(c)N\hnu-\frac{N-1}{\sqrt{\rho_0}}\int_{c}^{u_0}\left(\frac{1}{F(s)}\int_0^{s}\sqrt{\frac{{F(t)}}{{F(s)}}}\,\mathrm{d}t\right)\mathrm{d}s+o_{\eps}(1)\right)
			\end{align*}
		}
		to obtain the third equality of \eqref{0803-343}, and the last equality of \eqref{0803-343} is verified due to \eqref{mu-asy} and
		\begin{align*}
		\Phi^{-1}(c)=(\Psi^R)^{-1}(c)=\sqrt{\alpha(N)}R^{N/2}\Psi^{-1}(c)\,\,(\mathrm{cf.}\,\,\eqref{tmxi}).
		\end{align*}

		We shall stress that \eqref{0803-343} is obtained from \eqref{0718-528}, in which $C_{10}(j,{R})$ with $j=2\Phi^{-1}(c)=2\sqrt{\alpha(N)}R^{N/2}\Psi^{-1}(c)$ depends on $R^{N/2}$. Consequently, as $\eps\rightarrow0$, the convergence of \eqref{0803-343} is uniformly in $(0,R_0]$ for any $R_0>0$.  Therefore,  we obtain \eqref{1055-0728} and complete the proof of Lemma~\ref{lem-add8}.
	\end{proof}

	Now we are in a position to prove Theorem~\ref{mainthm}.
	
	\begin{proof}[Proof of Theorem~\ref{mainthm}.]
		Theorem~\ref{mainthm}-(i) immediately follows from Proposition~\ref{newt}. Next, let $d_0=0$ in \eqref{asys-u-diff}, we get \eqref{0820-slope} and complete the proof of Theorem~\ref{mainthm}-(ii).

		It remains to prove Theorem~\ref{mainthm}-(iii). First, we obtain \eqref{good-asy} following from \eqref{mu-asy} and \eqref{1055-0728}. Since ${\hnu}<0$ and $\Psi^{-1}(c)>0$ are independent of $\eps$ and $R$, \eqref{good-asy} implies
		\begin{align}\label{eps-0820}
		{R-r_{\eps}(R,c)}=\widehat{C}_1{\eps}R^{N/2}+\widehat{C}_2{\eps}^2(R^{N-1}+o_{\eps}(R)),
		\end{align}
		where $\widehat{C}_1=\sqrt{\alpha(N)}\Psi^{-1}(c)$ and
		\begin{align*}
		\widehat{C}_2=\frac{\alpha(N)}{2}\left(-\frac{N}{\sqrt{m}}\Psi^{-1}(c)\hnu+{\frac{N-1}{m}}\int_{c}^{u_0}\left(\frac{1}{F(s)}\int_0^{s}\sqrt{\frac{{F(t)}}{{F(s)}}}\,\mathrm{d}t\right)\mathrm{d}s\right)
		\end{align*}
		are positive constants independent of $\eps$ and $R$, and by Lemma~\ref{lem-add8}, $o_{\eps}(R)$ is continuously differentiable with respect to $R$ and satisfies
		\begin{align*}
		\lim_{\eps\rightarrow0}\sup_{R\in(0,R_0]}|o_{\eps}(R)|=0
		\end{align*}
		for any $R_0>0$. {Since both $\widehat{C}_1$ and $\widehat{C}_2$ are positive, we can choose $\e$ sufficiently small such that the derivative of the right hand side of \eqref{eps-0820} with respect to $R$ is positive. As a consequence, ${R-r_{\eps}(R,c)}$ is strictly increasing with respect to $R\in(0,R_0]$ for such $\eps$.} The proof of Theorem~\ref{mainthm} is thus completed.
	\end{proof}

	\section{Appendix}
	In this appendix, we will follow the arguments in \cite[Lemma 10.5]{pr} to give the proof of \eqref{3.1-fc}.
	\begin{lemma}
		\label{lea.1}
		The Euclidean Laplacian $\Delta$ can be computed by a formula in terms of the coordinate $(y,z)\in\mathcal{O}$ as
		\begin{equation*}
		\Delta_x=\partial_z^2-H_{\Gamma_z(y)}\partial_z+\Delta_{\Gamma_z},\quad
		x=X(y,z),\quad (y,z)\in\mathcal{O},
		\end{equation*}
		where $\Gamma_z$ is the manifold
		\begin{equation*}
		\Gamma_z=\left\{y+z\nu(y)\mid y\in\partial\Omega\right\},
		\end{equation*}
		and $H_{\Gamma_z(y)}$ is the mean curvature of $\Gamma_z$ measured at $y+z\nu(y).$
	\end{lemma}
	
	\begin{proof}
		For simplicity we only show the above formula when $z=0$. Let $e_1,\cdots,e_n$ be an orthonormal frame coordinate on $\partial\Omega$ and $\nu$ be the normal vector field.
		
		The Laplace-Beltrami operator on $\mathcal{O}$ is defined by
		\begin{equation*}
		\Delta_g=\sum_{i=1}^n(e_ie_i-D_{e_i}e_i)+\nu\nu-D_{\nu}\nu,
		\end{equation*}
		where $D$ is the Levi-Civita connection on $\mathcal{O}$. Let $D^{\partial\Omega}$ denote the Levi-Civita connection on $\Omega$, by construction, we have
		\begin{equation*}
		D_{e_i}e_i=D_{e_i}^{\partial\Omega} e_i+g(D_{e_i}e_i,\nu)\nu.
		\end{equation*}
		Therefore
		\begin{equation*}
		\Delta_g=\sum_{i=1}^n(e_ie_i-D_{e_i}^{\partial\Omega} e_i)+g(e_i,D_{e_i}\nu)\nu+\nu\nu-D_{\nu}\nu.
		\end{equation*}
		By definition $\nu\nu=\partial_z^2$ and $\nu=\partial_z$. Furthermore $D_\nu\nu=0$ and
		$$\sum_{i=1}^ng(e_i,D_{e_i}\nu)=-H_{\partial\Omega}(y),$$
		where $H_{\partial\Omega}$ is the mean curvature of $\partial\Omega$. Hence we finish the proof.
	\end{proof}

	\section*{Acknowledgements}
	C.-C. Lee is partially supported by the Ministry of Science and Technology of Taiwan under the grant 108-2115-M-007-006-MY2. The research of Z. Wang is supported by the Hong Kong RGC GRF grant No. PolyU 153032/15P (Project ID P0005368). W. Yang is supported by NSFC No.11801550 and NSFC No.11871470.

\end{document}